\documentclass[a4paper,reqno]{amsart}
\usepackage{graphicx}
\usepackage{amsmath}
\usepackage{amssymb}
\usepackage{amsthm}
\usepackage{bm} % better bold math symbol
\usepackage{tikz-cd}
\usepackage{tikz}
\usepackage{undertilde} 
\usepackage{hyperref,enumerate}
\usepackage[english]{babel}

\newtheorem{thm}{Theorem}[section]
\newtheorem{lem}[thm]{Lemma}
\newtheorem{prop}[thm]{Proposition}
\newtheorem{cor}[thm]{Corollary}
\theoremstyle{definition}
\newtheorem{defin}[thm]{Definition}
\newtheorem{nota}[thm]{Notation}
\newtheorem{exam}[thm]{Example}

\newtheorem{assumption}[thm]{Assumption}
\theoremstyle{remark}
\newtheorem{rem}{Remark}
\newcommand{\homi}{\var{A}}
\newcommand{\func}{\mathcal{F}}
\newcommand{\pr}{\mathrm{pr}}
\newcommand{\deck}{\mathrm{deck}}
\newcommand{\im}{\mathrm{Im}}
\newcommand{\N}{\mathbb{N}}
\newcommand{\divi}{\mathrm{div}}

\newcommand{\cate}[1]{\mathcal{#1}}
\newcommand{\alg}[1]{\mathrm{\mathbf #1}} 
\newcommand{\struc}[1]{\langle #1 \rangle}
\newcommand{\classop}[1]{\mathbb{#1}}
\newcommand{\var}[1]{\mathcal{#1}}
\newcommand{\me}{\mathrm{\bf m}}
\newcommand{\ess}{\mathrm{ess}}
\newcommand{\Inv}{\mathrm{\mathbf{Inv}}}
\newcommand{\id}{\mathrm{id}}
\newcommand{\fin}{\mathrm{fin}}

\newcommand{\mbf}[1]{\boldsymbol{#1}}
\newcommand{\lucas}{\text{\bf\L}}
\newcommand{\free}{\mathfrak{F}}

\makeatletter
\@namedef{subjclassname@2020}{%
  \textup{2020} Mathematics Subject Classification}
\makeatother

\title[The Minor order. of homo. via Nat. Dual.]{The Minor Order of Homomorphisms via Natural Dualities}
\author{Wolfgang Poiger}
\address{Department of Mathematics, FSTM, University of Luxembourg, 6, Avenue de la Fonte, L-4364 Esch-sur-Alzette, Luxembourg}
\email{wolfgang.poiger@uni.lu}

\author{Bruno Teheux}
\address{Department of Mathematics, FSTM, University of Luxembourg, 6, Avenue de la Fonte, L-4364 Esch-sur-Alzette, Luxembourg}
\email{bruno.teheux@uni.lu}



\begin{document}

\begin{abstract}
We study the minor relation for algebra homomorphims in finitely generated quasivarieties that admit a logarithmic natural duality. We characterize the minor homomorphism posets of finite algebras in terms of disjoint unions of dual partition lattices and investigate reconstruction problems for homomorphisms.
\end{abstract}
\maketitle

\section{Introduction}\label{Introduction}

\emph{Minors} of an $n$-ary operation $f$ on a set $A$ are those operations $g$ on $A$ that can be obtained from $f$ by identifying or permuting arguments, or by adding/deleting inessential arguments. We write $g \preceq f$ if $g$ is a minor of $f$. Thus defined, the minor relation $\preceq$ is a preorder on the set of all operations on $A$, and the \emph{minor poset of $A$} is the corresponding partial order obtained after identifying operations by the \emph{minor equivalence} $g\equiv f \Leftrightarrow (g \preceq f \wedge f \preceq g)$. The minor preorder and its associated partial order were introduced in the setting of clone theory, where they were used to study and characterize equational classes of Boolean functions (see \cite{Ekin2000}, for instance). Since then, they were the topic of several investigations (see \cite{Bouaziz2010, Couceiro2005, Couceiro2018,   MinorPosets, Pippenger2002}, to name a few). The general idea behind these investigations is to determine how much information about an operation $f$ can be retrieved from its minors. In particular, a series of recent papers deals with reconstruction properties, asking which operations $f$ can, up to minor equivalence, be recovered from (a portion of) their minors (see \cite{Lehtonen2014,Lehtonen2015,Lehtonen2016,Reconstruction}). We adopt the more general framework of these papers, where not only operations but arbitrary functions of several arguments $f\colon A^n\to B$ are considered.

In this paper, we investigate the minor order in the particular case of algebra homomorphisms. More precisely, let $\var{A}$ be a class of algebras of the same type and for $\alg{A},\alg{B}\in \var{A}$ let $f\colon \alg{A}^n \to \alg{B}$ be a homomorphism. Then, any minor of $f$ is itself a homomorphism and we are naturally led to study the restriction of the minor poset to the collection of homomorphisms $\bigcup\{\var{A}(\alg{A}^n, \alg{B})\mid n\geq 1\}$ up to minor equivalence, which we call the \emph{minor homomorphism poset}. In that perspective, our main tool is the use of natural dualities to translate various problems about the minor homomorphism poset into their dual equivalents. If the duality on $\var{A}$ is nice enough (logarithmic, in our case), these dual problems turn out to be easier than the original ones. In particular, we solve problems about reconstructibility (Section \ref{sect:ideals}) and about finding structural descriptions of the minor homomorphism posets over finite algebras (Section \ref{MinorHomomorphism}). One purpose of this paper is to demonstrate how to successfully apply the theory of natural dualities and thus contribute to the popularization of these methods in combinatorics. 

The theory of natural dualities emerged in the late 1970s in order to give a common ground for the development of dual equivalences generalizing Stone duality for Boolean algebras and Priestley duality for distributive lattices. The general idea is that, for a finite algebra $\underline{\alg{M}}$, it is sometimes possible to find a discrete structure $\utilde{\alg{M}}$ based on the same set as $\underline{\alg{M}}$ (therefore often called an \emph{alter-ego} of $\underline{\alg{M}}$), such that the category of algebras $\var{A} = \classop{ISP}(\underline{\alg{M}})$ with homomorphisms is dually equivalent to the category of structured topological spaces $\var{X} = \classop{IS}_c\classop{P}(\utilde{\alg{M}})$ of closed subspaces of powers of $\utilde{\alg{M}}$ with continuous structure preserving maps. Simply put, under these circumstances every algebra $\alg{A}\in \var{A}$ has a dual topological structure $\alg{A}^* \in \var{X}$ and every homomorphism $\alg{A}\to\alg{B}$ has a dual morphism $\alg{B}^* \to \alg{A}^*$. We build our developments on previous approaches of clone theory in the framework of natural dualities (see \cite{Kerkhoff2011,Kerkhoff2013,EssVar}). 

Products and coproducts are interchanged under a dual equivalence, so the dual of a homomorphism $f\colon \alg{A}^n \rightarrow \alg{B}$ is a morphism $f^*\colon \alg{B}^* \to \coprod_n \alg{A}^*$.
We restrict our investigation to quasivarities $\var{A}=\classop{ISP}(\underline{\alg{M}})$ for which there is an alter-ego $\utilde{\alg{M}}$ that yields a \emph{logarithmic} duality, which means that finite coproducts in the dual category of $\var{A}$ are realized by direct unions (disjoint unions with constants amalgamated). Under these circumstances, a dual morphism $f^*$ as above is easier to work with than $f$.

Let us point out some of our main achievements.  
Theorem \ref{thm:weakly} states that, up to minor equivalence, homomorphisms $f\colon \alg{A}^n \to \alg{B}$ without inessential arguments are determined by their identification minors (which are the homomorphisms obtained by identifying two arguments of $f$). Proposition \ref{prop:tadam} states that homomorphisms $f\colon \alg{A}^n \to \alg{B}$ are totally asymmetric in the sense that they have trivial invariance group.

Theorem \ref{thm:weakly} and Proposition  \ref{prop:tadam} build on Proposition \ref{partition}, which shows that the principal ideal generated in the minor homomorphism poset by a homomorphism $f\colon \alg{A}^n\to \alg{B}$ with $n$ essential arguments is anti-isomorphic to the full $n$-element partition lattice. Theorem \ref{MorphismsCharacterization} and its Corollary \ref{Top-Down} completely characterize the minor homomorphism posets of finite members of $\var{A}$ in terms of disjoint unions of such partition lattices, by identifying their maximal elements. 

We illustrate our developments with numerous examples. In particular, we show that even though the minor homomorphism poset of an algebra encodes very little information about this algebra, it is sometimes  enough to characterize some of the algebra's properties. For instance, we show that it is possible to recognize finite complemented lattices among distributive lattices by looking at their minor homomporphism posets (Proposition \ref{prop:complat}). 

The paper is organized as follows. In Section \ref{Preliminaries} we give a brief introduction to the theory of natural dualities and we introduce the examples of dualities that accompany us throughout the paper. We then recall the definition of the minor preorder and the constructions related to it. In Section \ref{section:minor}, we introduce the minor homomorphism posets and develop the techniques to investigate them through duality. In particular, we show how and why the setting of logarithmic dualities is especially well suited for the investigation of these posets. Section \ref{sect:ideals} is devoted to the description of principal ideals in minor homomorphism posets and to reconstructibility problems. In the finite case, we show the importance of maximal elements in the description of these posets. Therefore, Section \ref{MinorHomomorphism} focuses on identifying these maximal elements to characterize the minor homomorphism posets in terms of disjoint unions of partition lattices. We conclude the paper with final remarks and topics for further research.

%which is the field of universal algebra that investigate those sets of operations that are closed under function composition. 

\section{Preliminaries}\label{Preliminaries}

In this section, we set some notation and vocabulary for the rest of the paper. We recall the basic constructions of natural dualities, which we illustrate with a number of examples. In Subsection \ref{Minor} we recall the definition of the minor relation and the concepts related to it.   

\subsection{Posets, Partitions and Permutations}\label{sec:nota}
If $\alg{P}=\struc{P, \leq}$ is a partially ordered set (in short, a \emph{poset}) and $p\in P$, then $p{\downarrow}$ and $p{\uparrow}$ denote the \emph{principal ideal} and the \emph{principal filter} generated by $p$, respectively. We use $\alg{P}^\partial$ to denote the dual poset of $\alg{P}$, that is, the poset  $\struc{P, \leq^\partial}$  defined by $p\leq^\partial q$ iff $q\leq p$.

Recall that a poset $\alg{L}$ is called a \emph{lattice} if every pair $\{a,b\}$ of elements of $\alg{L}$ has a greatest lower bound $a\wedge b$ and a least upper bound $a\vee b$. A lattice $\alg{L}$ is \emph{distributive} if it satisfies the equation $x\wedge (y \vee z) = (x\wedge y) \vee (x\wedge z)$ and its dual. Distributive lattices can be characterized as those lattices which neither contain the pentagon lattice $N_5$ nor the diamond lattice $M_3$ as sublattice (see, \emph{e.g.}, \cite[Section 3.II.1]{GenLattice}). 

If $\pi$ is a partition of a set $S$, we often refer to its elements as \emph{blocks}.  We denote by $\Pi_n$ the set of partitions of $[n] := \{ 1,2,\dots, n \}$. In particular, we have $\Pi_0 = \{ \varnothing \}$ and $\Pi_1 = \{ \{ 1 \} \}$. We use the usual lattice order on $\Pi_n$, which is defined by  $\pi_1 \leq \pi_2$ iff for every $ B \in \pi_1$ there is some $C\in \pi_2$ such that $B \subseteq C$. We refer to $\struc{\Pi_n, \leq}$ as  the \textit{$n$-th partition lattice} $\alg{\Pi}_n$. The cardinality of  $\Pi_n$ is given by the \textit{Bell number} $B_n$ for every $n\geq 0$ (see \cite{OEIS:bell}).  

Finally, we denote by $\alg{S}_n$ the \emph{$n$-th symmetric group}, that is, the group of permutations over $[n]$.

\subsection{Natural Dualities}\label{sec:natdual}
The theory of natural dualities emerged in the late 1970s in order to give a common framework to develop and study dual equivalences for categories of algebras, generalizing Stone duality for Boolean algebras and Priestley duality for distributive lattices. Natural dualities are at the core of our investigation of the minor posets of homomorphims. Here we only recall the basic definitions of this theory. We refer the reader to \cite{NatDual} for a more detailed reference.

%Let $\mathcal{A}$ and $\mathcal{X}$ be two categories. A \textit{dual adjunction} between $\mathcal{A}$ and $\mathcal{X}$ is given by $(D,E,e,\varepsilon)$ where $D\colon \mathcal{A} \rightarrow \mathcal{X}$ and $E\colon \mathcal{X} \rightarrow \mathcal{A}$ are adjoint contravariant functors and $e\colon id_\mathcal{A} \Rightarrow ED$ and $\varepsilon\colon id_\mathcal{X} \Rightarrow DE$ are natural transformations. A dual adjunction is a \textit{dual representation} if $e_A \colon A \rightarrow ED(A)$ is an isomorphism for every $A\in \mathcal{A}$ and a \textit{dual equivalence} if, furthermore, $\epsilon_X \colon X \rightarrow DE(X)$ is an isomorphism for every $X\in\mathcal{X}$.

Let  $\mathcal{A} = \mathbb{I}\mathbb{S}\mathbb{P}(\underline{\textbf{M}})$ be the quasivariety generated by a finite algebra $\underline{\textbf{M}}$ with underlying set $M$ (we claim that our results can be naturally generalized for quasivarieties generated by a finite set of finite algebras). We denote by $\utilde{\alg{M}}$ an alter-ego of $\underline{\alg{M}}$, \emph{i.e.}, a topological structure
\[
\utilde{\alg{M}}=\struc{M, G, H, R, \mathcal{T}_{dis}},
\]
where $\mathcal{T}_{dis}$ is the discrete topology on $M$ and $G$, $H$, $R$ are a set (possibly empty) of algebraic operations, algebraic partial operations (with nonempty domain), and algebraic (nonempty) relations on $\underline{\alg{M}}$, respectively (here, an $n$-ary relation is algebraic on $\underline{\alg{M}}$ if it is a subalgebra of $\underline{\alg{M}}^n$, and an operation or a partial operation is algebraic on $\underline{\alg{M}}$ if its graph is algebraic on $\underline{\alg{M}}$). We denote by $\cate{X}$ the class $\classop{IS}_c\classop{P}(\utilde{\alg{M}})$ of topological structures that are isomorphic to a closed substructure of a nonempty power of $\utilde{\alg{M}}$, and we consider $\cate{X}$ as a category with continuous structure preserving maps as arrows. 
%For any $X, Y \in \cate{X}$ we denote by $\cate{X}(X,Y)$ the set of the structure preserving continuous maps $f:X\to Y$. 
We let $\mathcal{G}$, $\mathcal{H}$ and $\mathcal{R}$ be sets of symbols for operations, partial operations and relations, respectively, that correspond to the type of elements of $\cate{X}$.
For any $\alg{X} \in \cate {X}$, we use $\alg{X}_*$ to denote  $\cate{X}(\alg{X},\utilde{\alg{M}})$, considered as a subalgebra of $\underline{\alg{M}}^X$.

For every $\alg{A}\in \var{A}$, the Preduality Theorem \cite[Theorem 1.5.2]{NatDual} states that $\var{A}({\alg{A}, \underline{\alg{M}}})$ is a closed substructure $\alg{A}^*$ of $\utilde{\alg{M}}^A$, and therefore an element of $\cate{X}$. Moreover, the object mappings $\alg{A}\mapsto \alg{A}^*$ and $\alg{X} \mapsto \alg{X}_* $ can be lifted to contravariant functors $\cdot^*\colon \cate{A} \to \cate{X}$ and $\cdot_* \colon \cate{X} \to \cate{A}$ by setting
\begin{align*}
f^*(u)=u\circ f \qquad &\text{ for all } f\in\cate{A}(\alg{A}, \alg{B}) \text{ and }u \in \alg{B}^*,\\
\varphi_*(x)=x\circ \varphi \qquad & \text{ for all } \varphi \in \cate{X}(\alg{X}, \alg{Y}) \text{ and } x \in \alg{Y}_*.
\end{align*}
The Dual Adjunction Theorem \cite[Theorem 1.5.3]{NatDual} asserts that $\cdot^*$ and $\cdot_*$ define a dual adjunction between $\cate{A}$ and $\cate{X}$, where the associated natural transformations $e\colon 1_\cate{A} \to (\cdot^*)_*$ and $\varepsilon\colon 1_{\cate{X}} \to (\cdot_*)^*$ are given by
\begin{align*}
e_\alg{A}(a)(u)=u(a) \qquad & \text{for all } \alg{A}\in \var{A}, u \in \alg{A}^* \text{ and } a \in A,\\
\varepsilon_\alg{X}(\varphi)(x)=\varphi(x) \qquad & \text{for all } \alg{X}\in \cate{X}, \varphi \in \alg{X}_* \text{ and } x \in X.
\end{align*}

\begin{defin}[\cite{NatDual}]
We say that $\utilde{\alg{M}}$ \emph{yields a duality on $\var{A}$} if $e_A$ is an isomorphism for every $\alg{A}\in \var{A}$, and that $\utilde{\alg{M}}$ \emph{yields a full duality on $\var{A}$} if in addition $\varepsilon_\alg{X}$ is an isomorphism for every $\alg{X} \in \cate{X}$. A full duality is called \emph{strong} if $\utilde{\alg{M}}$ is injective in $\cate{X}$. 
\end{defin}

For $\alg{X}$ in $\mathcal{X}$, we denote by $\mathcal{C}^\alg{X}$ the set of $0$-ary functions in $\mathcal{G}^\alg{X}$  and refer to them as \textit{constants}, where each $c\in\mathcal{C}^{\alg{X}}$ is identified with its value, which forms a one-element subalgebra of $\underline{\textbf{M}}$. The structure on $\utilde{\textbf{M}}$ shall always be chosen such that no other total or partial function in $(\mathcal{G}^{\alg{X}}{\setminus} \mathcal{C}^{\alg{X}} \cup \mathcal{H}^{\alg{X}})$ is constant. %Furthermore, if $\utilde{\textbf{M}}$ yields a full duality on $\var{A}$ then $\mathcal{C}^{\alg{X}}$ is always a closed substructure of $\textbf{X}$.

 One of the main benefits of a full duality is that it maps products in one category to coproducts in the other category and vice versa. As we shall see throughout the paper, this correspondence is the key ingredient to our investigation of the minor order on homomorphisms.  
 
 The case of a full duality generated by a unary alter-ego  $\utilde{\alg{M}}$  (\emph{i.e.}, the partial and total operations of $\utilde{\alg{M}}$ are at most unary) is of particular interest, since  coproducts in $\cate{X}$ might turn out to be direct unions. Recall that the \emph{direct union} $\textbf{X}\oplus\textbf{Y}$ of $\alg{X},\alg{Y}\in \classop{IS}_c\classop{P}(\utilde{\alg{M}})$ (where $\utilde{\alg{M}}$ is unary) is defined on the disjoint union ${X}\uplus{Y}:=(\{1\}\times X) \cup (\{2\}\times Y)$ of $X$ and $Y$ by amalgamating $(1,c^\alg{X})$ and $(2,c^\alg{Y})$ for every $c\in \mathcal{C}$, by defining the operations and relation as the unions of the corresponding ones in $\alg{X}$ and $\alg{Y}$, and by equipping the resulting structure with the final topology with respect to the inclusion maps $\alg{X},\alg{Y} \to X\uplus Y$. If $\cate{X}$ is closed under direct unions, then the coproduct of $\alg{X}$ and $\alg{Y}$ in $\cate{X}$ is realized by $\alg{X}\oplus \alg{Y}$ (see \cite[Lemma 6.3.2]{NatDual}).
 
 \begin{defin}\label{defn:log}
 A unary structure $\utilde{\alg{M}}$ yields a \emph{logarithmic duality} on $\var{A}$ if it yields a strong duality on  $\var{A}$ and finite coproducts in $\cate{X}$ are realized by direct unions.
 \end{defin}
 
For the purposes of this paper it is convenient to think of the carrier set of the direct union of $\alg{X}$ and $\alg{Y}$ in a slightly different (but isomorphic) way as
\[
X{\setminus}\mathcal{C}^\alg{X} \uplus  Y{\setminus}\mathcal{C}^\alg{Y} \uplus \mathcal{C}^{\alg{X}\oplus\alg{Y}}.
\]   
The Logarithmic Duality Theorem \cite[Theorem 6.3.3]{NatDual} provides a sufficient condition to get a logarithmic duality. Here, a $n$-ary relation $R$ on $\utilde{\alg{M}}$ with arity $n \geq 2$ \emph{avoids binary products} if for all $1\leq i < j \leq n$ the set 
$$\{ \bigl(\pi_i\left(r\right),\pi_j\left(r\right) \bigr) \mid r\in R \}$$ 
contains no product of nontrivial subalgebras of $\utilde{\alg{M}}$.     

\begin{thm}[\cite{NatDual}]\label{thm:LDT}
Let $\utilde{\alg{M}}$ be a unary structure which yields a strong duality on $\mathbb{I}\mathbb{S}\mathbb{P}(\underline{\alg{M}})$. Then the following are equivalent:
\begin{enumerate}[(i)]
\item $\utilde{\alg{M}}$ yields a logarithmic duality on $\mathbb{I}\mathbb{S}\mathbb{P}(\underline{\alg{M}})$. 
\item For all $n\geq 2$, every $n$-ary relation of $\utilde{\alg{M}}$ avoids binary products. 
\end{enumerate}
\end{thm}

\begin{rem}
The main results of this paper are all based on the assumption of a logarithmic duality (see Assumption \ref{ass:main}). This framework might seem a bit narrow. However, due to the above theorem combined with other results from \cite{NatDual}, we can find a lot of examples of logarithmic dualities. In particular, there is a logarithmic duality for $\classop{ISP}(\underline{\alg{M}})$ if $\underline{\alg{M}}$ is \emph{quasi-primal}, that is, the ternary discriminator 
$$ t(x,y,z) = \begin{cases}
z & \text{ if } x = y \\
x & \text{ if } x \neq y
\end{cases} $$
is term-definable in $\underline{\alg{M}}$ (quasi-primal algebras are precisely the finite discriminator ones). Examples of quasi-primal algebras can, for example, be found in \cite{Werner1978, Burris1992}. In \cite{Murskii1975} it is shown that, over a fixed algebraic type containing some operation of arity at least $2$, almost all finite algebras of that type are quasi-primal. That is, a randomly chosen algebra of that type is quasi-primal with probability one.        
\end{rem}    

\subsection{Examples of Natural Dualities}\label{Examples}
We end this section with some concrete examples of full natural dualities, which will be used throughout this paper to both motivate and illustrate the results. All the dualities described here except for the last one are logarithmic due to the Logarithmic Duality Theorem \ref{thm:LDT}.

\subsubsection{Boolean Algebras.} Let $\mathcal{B}$ be the variety of Boolean algebras, which is generated as a quasivariety by the two-element Boolean algebra 
\[\underline{\alg{2}} := \struc{\{ 0,1\}, 0, 1, \wedge, \vee, \cdot^{-1}}.\]
The discrete space 
\[\utilde{\alg{2}} := \struc{\{ 0,1 \}, \mathcal{T}_{dis}}\]
 yields a strong duality between $\mathcal{B}$ and the category  $\mathcal{S} = \mathbb{I}\mathbb{S}_{c}\mathbb{P}(\utilde{\alg{2}})$ of Stone spaces (that is, zero-dimensional compact Hausdorff spaces). This is the formulation of the renowned \emph{Stone duality} (see \cite{Stone1936}) in the language of natural dualities. 
%The dual of a Boolean algebra $\textbf{B}\in\mathcal{B}$ is given by $\mathcal{B}(\textbf{B},\underline{2})$, which can be identified with the set of ultrafilters on $\textbf{B}$ since for every homomorphism $h\colon B \rightarrow \underline{2}$ the preimage $h^{-1}(\{ 1\})$ forms an ultrafilter. The evaluation $e_\textbf{B}(b)$ can then be identified with the set of all ultrafilters containing $b$, and the collection of all those clopen sets forms a basis for the topology $\tau$ on $\textbf{B}^*$.   

Any finite Boolean algebra $\alg{2}^k$ with $k\geq 1$ has the discrete space $\struc{[k],\mathcal{T}_{dis}}$ as dual space. The category of finite Boolean algebras $\mathcal{B}_{\fin}$ is therefore dually equivalent to the category  of finite sets.

\subsubsection{Distributive Lattices.}\label{sub:dl} Let $\mathcal{D}$ be the variety of (unbounded) distributive lattices, which is generated as a quasivariety by the two-element distributive lattice 
\[\underline{\alg{2}} = \struc{\{ 0,1 \}, \wedge, \vee}.\] 
The discrete  structure
\[\utilde{\alg{2}} := \struc{\{ 0,1 \}, 0,1, \leq, \mathcal{T}_{dis}}\]
yields a strong duality between $\mathcal{D}$ and the category $\mathcal{P}_{01} = \mathbb{I}\mathbb{S}_c\mathbb{P}(\utilde{\alg{2}})$ of \textit{bounded Priestley spaces}, \emph{i.e}, bounded ordered compact spaces $(X,0,1,\leq,\mathcal{T})$ in which  for all $x,y \in X$ with $x\not\leq y$ there is a clopen downset that contains $y$ but not  $x$. This is the formulation of the renowned \emph{Priestley duality} (see \cite{Priestley1970}) in the language of natural dualities . 

%For $\textbf{D}\in \mathcal{D}$ the dual $\textbf{D}^*$ can be identified with the set of prime ideals of $\textbf{D}$ ordered by inclusion, with a subbasis given by the sets of the form $e_D (d) = \{ I \in D^* \mid d\notin I \}$ and their complements. 

The full subcategory $\cate{D}_{\fin}$ of $\cate{D}$ consisting of finite distributive lattices is dually equivalent to the category of finite bounded posets (this discrete version of the Priestley duality is known as Birkhoff duality). The dual of a finite distributive lattice $\alg{D}$ can be equivalently constructed as the poset $\struc{J(\alg{D})_{01}, \leq_\alg{D}}$ of  the join-irreducible elements $J(\alg{D})$ of $\alg{D}$ with additional bounds $0$ and $1$, and the map $a\mapsto a{\downarrow} \cap J(\alg{D})$ is an isomorphism between $\alg{D}$ and the lattice of downsets of $\struc{J(\alg{D})_{01}, \leq_\alg{D}}$.

%{\color{red};;;;The original Birkhoff duality does not need to add bounds. One again constants are creating troubles;;;;}

\subsubsection{Median Algebras.} A \textit{median algebra} (see \cite{Birkhoff1947}) is a ternary algebra $\textbf{A} = \struc{A, \me}$ that satisfies the equations
\begin{gather*}
\me(x,x,y) = x,\\
 \me(x,y,z) = \me(y,x,z) = \me(y,z,x),\\
 \me(\me(x,y,z), v, w) =
 \me(x, \me(y,v,w), \me(z,v,w)). 
\end{gather*}

In particular, every distributive lattice $\textbf{D}\in\mathcal{D}$ yields a  median algebra $\struc{D, \me_\alg{D}}$ by stipulating
\[
\me_\alg{D} (a,b,c) = (a \wedge b) \vee (a \wedge c) \vee (b \wedge c). 
\]
It turns out that median algebras are exactly the subalgebras of median algebras $\struc{D, \me_\alg{D}}$ stemming from some $\alg{D}\in \var{D}$ (see \cite{Bandelt1983} and the references therein).

The variety of median algebras $\var{M}$ is generated as a quasivariety by the two element median algebra 
\[\underline{\alg{2}} := \struc{\{ 0,1 \}, \me},\] 
where $\me$ is the majority operation $\me(x,x,y)= \me(x,y,x) = \me(y,x,x)=x$. The discrete structure
\[
\utilde{\alg{2}} := \struc{\{ 0,1 \}, 0,1, \leq, \cdot^c, \mathcal{T}_{dis}},
\]
where $\leq$ is the natural order and $\cdot^c$ is the unary operation that swaps $0$ and $1$, is known to yield a strong duality on $\var{M}$ (see \cite{NatDual,Isbell1980,Werner1981}).
The dual category $\mathbb{I}\mathbb{S}_c\mathbb{P}(\utilde{\alg{2}})$ is the category of \emph{bounded strongly complemented Priestley spaces}, that is, bounded Priestley spaces with an order-reversing homeomorphism $\cdot^c$ which is an involution and that satisfies
$$ x\leq x^c \implies x = 0. $$ 
%{\color{red} $\mathbb{I}\mathbb{S}_c\mathbb{P}^+$ is sometimes used, and $ \mathbb{I}\mathbb{S}_c\mathbb{P}$ is sometimes used}

%In the finite case we get a strong duality between $\mathcal{M}_{fin}$ and the category of finite bounded strongly complemented posets.  

\subsubsection{MV${}_m$-algebras} Let $\struc{[0,1], \oplus, \odot, \neg, 0, 1}$ be the standard MV-algebra defined by
\[
x \oplus y=\min(1, x+y), \quad 
x \odot y= \max(0, x+y-1), \quad \neg x= 1-x.
\]
The variety $\var{MV}_m$ of MV${}_m$-algebras (where $m>0$) is defined as $\var{MV}_m:=\classop{ISP}(\underline{\lucas}_m)$ where $\underline{\lucas}_m$ is the subalgebra $\{0, \frac{1}{m}, \ldots, \frac{m-1}{m}, 1\}$ of $[0,1]$. MV${}_m$-algebras are the algebras of \L ukasiewicz $(m+1)$-valued logic. It is known (see \cite{Niederkorn2001}) that the discrete structure with unary relations
\begin{equation*}\label{eqn:duallucas}
\utilde{\lucas}_m:=\struc{\lucas_m, \{\lucas_d \mid d \in \divi(m)\}, \mathcal{T}_{dis}},
\end{equation*}
where $\divi(m)$ is the set of positive divisors of $m$, yields a logarithmic strong duality for $\var{MV}_m$.

\subsubsection{Boolean Groups.}\label{ex:boolgroups} A \emph{Boolean group} is a group $\struc{G, +, 0}$ in which every element $x \neq 0$ is of order two. The quasivariety of Boolean groups, denoted $\mathcal{BG}$, is generated by 
$$ \underline{\alg{Z}_2} = \struc{\{0,1\}, +, 0},$$
where $+$ is addition modulo $2$. The structure $\utilde{\alg{Z}_2} = \struc{\{0,1\},+,0 ,\mathcal{T}_{dis}}$
yields a full duality for $\mathcal{BG}$ (see \cite{Kerkhoff2011}). In particular, the full subcategory $\mathcal{BG}_{fin}$ of finite Boolean groups is self-dual. Since products and coproducts coincide in self-dual categories, this duality is not logarithmic.

\subsection{The Minor Relation}\label{Minor}
Let $A$ and $B$ be two nonempty sets and $n$ be a positive integer. A \textit{$n$-ary function from $A$ to $B$} is a function $f\colon A^n \rightarrow B$. The collection of all such functions is denoted by $\mathcal{F}_{AB}^{(n)}$ and the \textit{functions of several arguments from $A$ to $B$} are the elements of
\[
\mathcal{F}_{AB} := \bigcup_{n\geq 1} \mathcal{F}_{AB}^{(n)}.
\]
For $f \in \mathcal{F}_{AB}$ the \textit{arity of $f$} is the unique $n\in\N$ for which $f\in \mathcal{F}_{AB}^{(n)}$ and is denoted by $ar(f)$.
 
Every map $\tau\colon [n] \rightarrow [m]$ induces a map $\tau^A\colon A^m \rightarrow A^n$ via $\tau^A(a_1, \dots, a_m) = (a_{\tau(1)},\dots, a_{\tau(n)})$. For $g\in\mathcal{F}_{AB}^{(m)}$ and $f\in\mathcal{F}_{AB}^{(n)}$, we say that $g$ is a \textit{minor} of $f$ and write $g\preceq f$ if there is some $\tau\colon [n] \rightarrow [m]$ such that $g = f \circ \tau^A$.  
The relation $\preceq$ is a preorder on $\mathcal{F}_{AB}$ (see \cite[Subsection 2.2]{Reconstruction}) called the \emph{minor preorder}. As every preorder does, it induces an equivalence relation $\equiv$, called the \textit{minor equivalence} on $\mathcal{F}_{AB}$, given by $f \equiv g$ iff both $f\preceq g$ and $g \preceq f$. The \emph{minor order} is the partial order on $\mathcal{F}_{AB}/{\equiv}$ (well-)defined by $[g] \leq [f]$ iff $g \preceq f$. The \textit{minor $(A,B)$-poset} is given by 
$$\mbf{\mathcal{F}}_{\alg{A}\alg{B}}:=\struc{\mathcal{F}_{AB}/{\equiv} ,\leq}.$$
We will simply write $\mbf{\mathcal{F}_{\alg{A}}}$ for the minor $(A,A)$-poset (instead of $\mbf{\mathcal{F}_{\alg{A}\alg{A}}}$).   
 
For $f \in \mathcal{F}_{AB}^{(n)}$ and $i\in [n]$ we say that the $i$-th argument of $f$ is \textit{essential} if there are $a,b \in A^n$ with $a_i \neq b_i$ and $a_j = b_j$ for every $j \neq i$ such that $f(a) \neq f(b)$. Otherwise the $i$-th argument of $f$ is called \textit{inessential}. The number of essential arguments of $f$ is called the \textit{essential arity} of $f$, denoted by $\ess(f)$. By definition we always have $\ess(f) \leq ar(f)$. 

Informally, we have $g \preceq f$ if $g$ can be obtained from $f$ by permuting arguments, identifying arguments, or by adding/deleting inessential arguments. In particular, for every function $f \in \mathcal{F}_{AB}$ there is a function $f'$ equivalent to $f$ with $ar(f') = \ess(f')$. We usually choose such functions without inessential arguments as representatives for $\equiv$. If $f$ and $g$ are in $\mathcal{F}_{AB}^{(n)}$ and both have no inessential arguments, then $f\equiv g$ holds if and only if there is a permutation $\sigma\in \alg{S}_n$ with $f = g\circ \sigma^A$.

For $n\geq 2$ let $\binom{[n]}{2}$ be the family of $2$-element subsets of $[n]$. For $I = \{ i,j \} \in \binom{[n]}{2}$ with $i < j$, define the map $\delta_I \colon [n]\rightarrow [n-1]$ by

\begin{equation*}
\delta_I (k) = \begin{cases}
     k & \text{for } k < j  \\
     i & \text{for } k = j  \\
     k-1 & \text{for }  k > j.
   \end{cases}
\end{equation*}
Given $f\in \mathcal{F}_{AB}^{(n)}$, we write $f_I$ for the minor $f\circ \delta_I^A$ and  call it an \emph{identification minor} of $f$. The identification minor $f_{\{ i,j \}} \preceq f$ is the result of identifying the $i$-th and the $j$-th argument of $f$.

To introduce the next notion, recall that a \emph{multiset} is a collection of elements in which elements are allowed to appear more than once. Formally, a multiset is a pair $(M,m)$ where $M$ is a set and $m\colon M\rightarrow \mathbb{Z}^{+}$ assigns a multiplicity to each element of $M$. For example, we write $\{ a, a, a, b, b\}$ for the multiset $(\{a,b\},m)$ where $m(a) = 3$ and $m(b) = 2$. Clearly a multiset $(M,m)$ can be identified with a (regular) set if $m(x) = 1$ holds for all $x\in M$. 

For $f\in \mathcal{F}_{AB}^{(n)}$, the \emph{deck of $f$} is the multiset $deck(f) = \{ [f_I] \mid I\in \binom{[n]}{2} \}$ of all equivalence classes of identification minors of $f$. A function $g \in \mathcal{F}_{AB}^{(n)}$ is a \textit{reconstruction} of $f$ if $deck(f) = deck(g)$.

\begin{exam}
Let $f: \mathbb{Z}^3 \rightarrow \mathbb{Z}$ be given by $f(x,y,z) = x^2 + y^2 + z^2$. Then $f_{\{ 1,2 \}}(x,y) = 2x^2 + y^2$ and $f_{\{ 2,3 \}}(x,y) = x^2 + 2y^2$. Since these two identification minors only differ by a permutation of arguments, we have $f_{\{ 1,2 \}} \equiv f_{\{ 2,3 \}}$. Still, in the deck of $f$ we count the equivalence classes of these two identification minors separately, that is, $deck(f) = \{[f_{\{ 1,2 \}}],[f_{\{ 1,3 \}}], [f_{\{ 2,3 \}}] \}.$   
\end{exam}

The only case in which the deck of $f\in \mathcal{F}_{AB}^{(n)}$ can be identified with a (regular) set is if all the identification minors of $f$ are pairwise non-equivalent. As we will see later (in Proposition \ref{prop:tadam}), this is actually the case in our setting. 

A series of recent papers (see \cite{Lehtonen2014,Lehtonen2015,Lehtonen2016,Reconstruction}) deals with reconstruction properties in the following sense. 
 
\begin{defin}[{\cite{Lehtonen2014}}]\label{defn:recons}
A function $f\in \mathcal{F}_{AB}$ is \emph{reconstructible} if all of its reconstructions are equivalent. Furthermore, if $\mathcal{C}$ is a subclass of $\mathcal{F}_{AB}$ we say that 
\begin{itemize}
 \item $\mathcal{C}$ is \emph{reconstructible} if all members of $\mathcal{C}$ are reconstructible,
 \item $\mathcal{C}$ is \emph{weakly reconstructible} if for every $f\in \mathcal{C}$, all the reconstructions of $f$ which are members of $\mathcal{C}$ are equivalent, 
 \item $\mathcal{C}$ is \emph{recognizable} if all reconstructions of members of $\mathcal{C}$ are again members of $\mathcal{C}$. 
\end{itemize}  
\end{defin}  

\begin{rem}
The minor relations and reconstruction problems introduced in this section stem from corresponding relations and problems in Graph Theory, where they are topics of long-term investigation (see \cite{Diestel2017} for an introducion). 
\end{rem}

In Section 4 we show that certain classes of homomorphisms of sufficient arity are weakly reconstructible (see Theorem \ref{thm:weakly}). Before that, we `set the scene' of restricting the minor poset to homomorphisms. This is the purpose of the next section. 
            
\section{Minor Homomorphism Posets}\label{section:minor}

We begin the section by introducing the minor homomorphism posets, which are obtained by restricting the minor relations to those functions that are algebra homomorphisms. Then, we restrict our investigation of homomorphisms posets to quasivarieties $\var{A}=\classop{ISP}(\alg{M})$ for which there is a discrete structure $\utilde{M}$ that yields a logarithmic duality on $\var{A}$. We introduce co-minor relations for morphisms in the dual category $\cate{X}:=\classop{IS}_c\classop{P}$ (see Subsection \ref{sec:natdual}) and show that they correspond by duality to the minor relations for homomorphisms in $\var{A}$ (see Corollary \ref{cor:dualminor}).

This correspondence is central for the results stated in the paper. Indeed, our reconstruction result for homomorphisms (see Theorem \ref{thm:weakly}) and the structural analysis of the homomorphism posets that we carry out in Section \ref{MinorHomomorphism} are based on the correspondence between the minor and co-minor relations.

\subsection{Minors of homomorphisms}
We use the notation introduced in Subsection~\ref{Minor}. Furthermore, $\var{A}$ will always denote a category of algebras of the same type with homomorphisms. We will soon require additional assumptions for $\var{A}$ (see Assumption \ref{ass:main}) but for now we may keep this level of generality.  

\begin{defin}
For every $n\geq 1$ and every $\alg{A}, \alg{B} \in \var{A}$ we set $\homi_{\alg{A}\alg{B}}^{(n)}:=\var{A}(\alg{A}^n,\alg{B})$, and \[\homi_{\alg{A}\alg{B}}:=\bigcup_{n\geq 1} \homi^{(n)}_{\alg{A}\alg{B}}.
\]
We call $\mbf{\homi}_{\alg{A}\alg{B}}:=\struc{\homi_{\alg{A}\alg{B}}/{\equiv}, \leq}$ the \emph{minor $(\alg{A},\alg{B})$-homomorphism poset}. Instead of $\mbf{\homi}_{\alg{A}\alg{A}}$ we simply write $\mbf{\homi}_{\alg{A}}$.
\end{defin}
%We clearly have $\homi^{(n)}_{\alg{A}\alg{B}}\subseteq \func^{(n)}_{{A}{B}}$, and $\homi_{\alg{A}\alg{B}}\subseteq \func_{{A}{B}}$. 
Our first observation is that if $f \in \homi_{\alg{A}\alg{B}}$ and $g\preceq f$, then  $g\in \homi_{\alg{A}\alg{B}}$.  
\begin{lem}\label{lem:tau}
Let $\tau\colon [n] \to [m]$, and $\alg{A}, \alg{B}\in \var{A}$.
\begin{enumerate}
\item\label{it:kju01} The map $\tau^A\colon A^m \to A^n$ belongs to $\var{A}(\alg{A}^m,\alg{A}^n)$.
\item\label{it:kju02} If $f\in \homi^{(n)}_{\alg{A}\alg{B}}$ then $f\circ \tau^A \in \homi^{(m)}_{\alg{A}\alg{B}}$.
\item\label{it:kju03} If $f\in \homi_{\alg{A}\alg{B}}$ then $[f]\subseteq \homi_{\alg{A}\alg{B}}$ and $[f]{\downarrow} \subseteq \homi_{\alg{A}\alg{B}}/{\equiv}$.
\end{enumerate}
In particular, $\homi_{\alg{A}\alg{B}}/{\equiv}$ is a downset in $\struc{\func_{AB}/{\equiv}, \leq}$. 
\end{lem}

\begin{proof}
Let $O$ be a $k$-ary operation in the signature of $\alg{A}$, and $\mathbf{a}^1, \ldots, \mathbf{a}^k \in \alg{A}^m$. We obtain successively
\begin{align*}
\tau^A(O^{\alg{A}^m}(\mathbf{a}^1, \ldots, \mathbf{a}^k)) & =\tau^A (O^\alg{A}(a_1^1\dots, a_1^k),\dots , O^\alg{A}(a_m^1, \dots, a_m^k))\\
& =(O^\alg{A}(a_{\tau(1)}^1,\dots, a_{\tau(1)}^k ), \dots, O^\alg{A}(a_{\tau(n)}^1,\dots, a_{\tau(n)}^k)), \\
& = O^{\alg{A}^n}((a^1_{\tau(1)},\dots,a^1_{\tau(n)}), \dots, (a^k_{\tau(1)},\dots, a^k_{\tau(n)}))\\
& = O^{\alg{A}^n}(\tau^A(a_1^1,\dots, a^1_m), \dots, \tau^A(a_1^k, \dots, a_m^k))\\
& = O^{\alg{A}^n}(\tau^A(\mathbf{a}^1), \ldots, \tau^A(\mathbf{a}^m)),
\end{align*} 
which proves (\ref{it:kju01}). 

(\ref{it:kju02}) and (\ref{it:kju03}) immediately follow from (\ref{it:kju01}).
\end{proof}
\subsection{Dualizing the minor relation of homomorphisms}\label{subsec:dualizingminor} We show how to dualize the minor relation on $\var{A}_{\alg{A}\alg{B}}$ under the assumption that there is a logarithmic natural duality for $\var{A}$ (we use the notation of Subsection \ref{sec:natdual}). All our upcomping results about the minor relation on homomorphisms are based on Assumption \ref{ass:main}, that holds for the remainder of the paper. 

\begin{assumption}\label{ass:main}
$\var{A}$ is the quasivariety $\classop{ISP}(\underline{\alg{M}})$ generated by a finite algebra $\underline{\alg{M}}$ and the discrete structure $\utilde{\alg{M}}=\struc{M, {G}, {H}, {R}, \mathcal{T}_{dis}}$ yields a logarithmic duality on $\var{A}$. 
\end{assumption} 
We use $\cate{X}$ to denote the dual category of $\var{A}$. 
For an object $\alg{X}$ of $\mathcal{X}$ we denote by $n\textbf{X}$ the $n$-th copower of $\textbf{X}$. 
As noted after Definition \ref{defn:log}, for our purpose it is convenient to consider the carrier of $n\alg{X}$ as $n$ disjoint copies of $X{\setminus} \mathcal{C}^{\alg{X}}$ with the constants $\mathcal{C}^{n\alg{X}}$ added separately:
\[
nX = \big([n] \times (X{\setminus} \mathcal{C}^{\alg{X}})\big) \cup \mathcal{C}^{n\alg{X}}.
\]

We abbreviate $X {\setminus} \mathcal{C}^{\alg{X}}$ by $X^\flat$, and for every $i\leq n$ we will refer to the set $\{i\} \times X^\flat$ as the \emph{$i$-th copy of $X^\flat$ in $n\alg{X}$}.
 
% Now, given a homomorphism $f\colon \textbf{A}^n \rightarrow \textbf{B}$, we can look at the dual morphism $f^* \colon \textbf{B}^* \rightarrow n\textbf{A}^*$.

Given any $\alg{A}, \alg{B} \in \var{A}$, we aim to describe the dual of the minor preorder $\preceq$ on $\homi_{\alg{A}\alg{B}}$. For every $f\in \homi^{(n)}_{\alg{A}\alg{B}}$ and $g\in \homi^{(m)}_{\alg{A}\alg{B}}$ with $g\preceq f$, there is a map $\tau\colon [n] \to [m]$ such that the diagram
\begin{equation*}
\begin{tikzcd}A^m\arrow[r ,"\tau^A"]\arrow[dr, "g"]& A^n\arrow[d, "f"]\\& B\end{tikzcd}
\end{equation*}
commutes. According to Assumption \ref{ass:main} and the first statement of Lemma \ref{lem:tau}, the previous diagram is equivalent to
\begin{equation*}
\begin{tikzcd}mA^*\arrow[r , leftarrow, "(\tau^A)^*"]\arrow[dr, leftarrow, "g^*"]& nA^*\arrow[d, leftarrow, "f^*"]\\& B^*\end{tikzcd}
\end{equation*}
in the dual category $\cate{X}$. Hence, in order to translate the minor relation to $\cate{X}$, we need to characterize the dual of the map $\tau^{A}$ for $\tau\colon [n] \to [m]$. Lemma \ref{Lem1} states that   $(\tau^{A})^*$  identically maps the $i$-th copy of $\alg{A}^{*\flat}$ in $n\alg{A}^*$ to the $\tau(i)$-th copy of $\alg{A}^{*\flat}$ in $m\alg{A}^*$. 

\begin{defin}\label{defn:tautilde}
Let $\tau \colon [n] \rightarrow [m]$ and $\alg{X}\in\cate{X}$. The \emph{term-wise identity map induced by $\tau$  on $\alg{X}$ } is the map $\tau_\alg{X}\colon nX \rightarrow mX$ defined by $\tau_{\alg{X}}(c^{n\alg{X}}) = c^{m\alg{X}}$ for all $c\in\mathcal{C}$, and  
\begin{equation*}
\tau_\alg{X}(i, x) = (\tau(i) , x) \qquad \text{ for all } i\in [n], x\in X^\flat.  
\end{equation*} 
\end{defin}

\begin{lem}\label{Lem1}
Let $\tau \colon [n] \rightarrow [m]$. For every $\alg{A} \in \var{A}$ we have $(\tau^A)^*=\tau_{\alg{A}^*}$.  
\end{lem}

\begin{proof}
By Assumption \ref{ass:main}, for every $k\geq 1$ the map $\varphi_k\colon k\alg{A}^* \to (\alg{A}^k)^*$ defined by $\varphi_k(c^{k\alg{A}^*})=c^{(\alg{A}^k)^*}$ for every $c\in \mathcal{C}$ and $\varphi_k\big((i,u)\big)=u \circ \pr_i$ for every $i\in [k]$ and $u\in\alg{A}^{*\flat}$ is an isomorphism in $\var{X}$. Hence, the statement of the lemma is equivalent to
\begin{equation}\label{eqn:bgt}
(\tau^A)^*\circ \varphi_n=\varphi_m \circ \tau_{\alg{A}^*}.
\end{equation}
Let $u\in \alg{A}^*$ and $a \in \alg{A}^m$. On the one hand, we successively obtain  
\begin{align*}
\Big(\big((\tau^{A})^*\circ\varphi_n\big) (i,u)\Big) (a) &=\big((\tau^{A})^*(u \circ \pr_i)\big)(a)\\
& = \big((u\circ \pr_i) \circ \tau^{{A}}\big)(a)\\
&= u(a_{\tau(i)}).
\end{align*}
On the other hand, we successively obtain 
\begin{align*}
\big((\varphi_m \circ \tau_{\alg{A}^*})(i,u)\big)(a) &=\varphi_m\big((\tau(i), u)\big)(a)\\
& = (u\circ \pr_{\tau(i)})(a)\\
& = u(a_{\tau(i)}).
\end{align*}
Thus we have verified identity \eqref{eqn:bgt}.  
\end{proof}

Lemma \ref{Lem1} together with the argument preceding Definition \ref{defn:tautilde} lead to the following definition.

\begin{defin} \label{cor:dualminor}
Let $\alg{X}, \alg{Y} \in \cate{X}$ and $\varphi\in\mathcal{X}(\textbf{Y}, n\textbf{X})$, $\psi \in \mathcal{X}(\textbf{Y},m\textbf{X})$. We say that $\psi$ is a \emph{co-minor} of $\varphi$, and we write $\psi \preceq_d \varphi$, if there is a map $\tau\colon [n] \rightarrow [m]$ such that $\psi= \tau_\alg{X} \circ \varphi$. 
\end{defin}
It is easy to check that $\preceq_d$ is a preorder on $\cate{X}_{\alg{Y}\alg{X}}:=\bigcup_{n\geq 1} \cate{X}(\alg{Y}, n\alg{X})$ for all $\alg{X}, \alg{Y}\in \cate{X}$, and we denote the equivalence relation associated with it by $\equiv_d$. Moreover, we denote by $\leq_d$  the partial order induced by $\preceq_d$ on $\cate{X}_{\alg{X}\alg{Y}}/{\equiv_d}$ and we set $$\mbf{\cate{X}}_{\alg{X}{\alg{Y}}}:=\struc{\cate{X}_{\alg{X}\alg{Y}}/{\equiv_d}, \leq_d}.$$  If $\varphi \in \cate{X}_{\alg{X}\alg{Y}}$, then we denote the class of $\varphi$ for $\equiv_d$ by $[\varphi]_d$. As usual, we write $\mbf{\cate{X}}_{\alg{X}}$ instead of $\mbf{\cate{X}}_{\alg{X}\alg{X}}$.

The cornerstone for our investigation of the minor homomorphism posets via duality is the following result.

\begin{cor} \label{cor:iso1}
Let $\alg{A}, \alg{B} \in \var{A}$ and $f,g\in \homi_{\alg{A}\alg{B}}$.
\begin{enumerate}
\item The map $\cdot^*\colon \struc{\homi_{\alg{A}\alg{B}}, \preceq} \to \struc{\cate{X}_{\alg{B}^*\alg{A}^*},\preceq_d}$ is an isomorphism of preorders.
\item The induced map $\cdot^*\colon \mbf{\homi}_{\alg{A}\alg{B}} \to \mbf{\cate{X}}_{\alg{B}^*\alg{A}^*}$ 
%$\cdot^*\colon \struc{\homi_{\alg{A}\alg{B}}/{\equiv}, \leq} \to \struc{\cate{X}_{\alg{B}^*\alg{A}^*}/{\equiv_d},\leq_d}$ 
defined by $[f]^*=[f^*]_d$ is a poset isomorphism.
\end{enumerate}
\end{cor}
To conclude this section, we recall how to recognize inessential arguments by duality from \cite{EssVar}. 

\begin{defin}
Let $\alg{X}, \alg{Y} \in \cate{X}$, let $\varphi \in \mathcal{X}(\textbf{Y}, n\textbf{X})$ and $i\in [n]$. We say that the $i$-th co-argument of $\varphi$ is \emph{essential} if
$$ \varphi (Y) \cap (\{ i \} \times X^\flat) \neq \varnothing.$$ The \emph{co-essential arity} of $\varphi$ is its number of essential co-arguments, denoted by $\ess_d(\varphi)$. 
\end{defin}

\begin{lem}[{\cite[Lemmas 3.4 and 3.9, Proposition 4.2]{EssVar}}]\label{lem:essvar}
Let $f\in \homi_{\textbf{A}\textbf{B}}^{(n)}$. For every $i\in [n]$ the following conditions are equivalent.
\begin{enumerate}[(i)]
 \item The $i$-th argument of $f$ is essential.
 \item The $i$-th co-argument of $f^*$ is essential.
\end{enumerate}
Therefore, $\ess(f) = \ess_d(f^*)$. 
\end{lem}

In other words, in order to determine the essential arguments of  $f\in \homi_{\alg{A}\alg{B}}^{(n)}$ we only need to ask which copies of $\alg{A}^{*\flat}$ intersect with $f^*(\alg{B}^*)$ in $n\alg{A}^*$. 

\section{Principal Ideals and Weak Reconstructibility}\label{sect:ideals}
Inspired by \cite{MinorPosets}, we look at principal ideals $[f]{\downarrow}$ in $\mbf{\homi}_{\alg{A}\alg{B}}$ and relate them to partition lattices. It turns out that, in our setting, every such principal ideal is anti-isomorphic to the partition lattice of size $\ess(f)$ (see Proposition \ref{partition}). The deck (as introduced in Subsection \ref{Minor}) of a homomorphism $f$ forms a diverse collection which seems to carry a lot of information about $f$. This suggests that $f$ is likely to be reconstructible. Indeed, as stated in Theorem \ref{thm:weakly}, this is the case if we only consider reconstructions which are themselves homomorphisms. 
        
Let $f\in\func_{{A}{B}}^{(n)}$ and $\pi \in \Pi_n$ be a partition with $m$ blocks. For every $\ell \in [n]$, denote by $\pi^\ell$ the block of $\pi$ that contains $\ell$. Any bijective labeling $c\colon\pi \to[m]$ defines a minor $f_c:=f\circ (\hat{c}^A)$, where $\hat{c}\colon[n]\to [m]$ is defined by $\hat{c}(\ell)=c(\pi^\ell)$. Moreover, we have $f_c\equiv f_{c_0}$ for any two bijective labelings $c,c_0\colon \pi \to [m]$. This justifies the following definition.

\begin{defin}\label{defn:fpi}
Let $f\in \func_{{A}{B}}^{(n)}$ and $\pi \in \Pi_n$ with cardinality $m$. We denote by $[f_\pi]$ the equivalence class of $f_c$ for $\equiv$, where $c\colon \pi \to [m]$ is any bijective labeling of the elements of $\pi$.  If in addition $f\in \homi_{\alg{A}\alg{B}}^{(n)}$, then we denote by $[f_\pi^*]_d$ the equivalence class for $\equiv_d$ of $(f_c)^*=\hat{c}_{\alg{A}^*}\circ f^*$.
\end{defin}

It is known that for every $f\in \func_{AB}^{(n)}$, the mapping $[f_{\cdot}]\colon\Pi_n \to [f]{\downarrow}$ defined by $\pi \mapsto [f_\pi]$ is  onto and order-reversing (see \cite[Corollary 7]{MinorPosets}). It may happen that $[f_\pi]=[f_{\pi'}]$ holds for distinct $\pi$ and $\pi'$ in $\Pi_n$, and \cite{MinorPosets} is devoted to the characterization of those equivalence relations $\sim$ on $\Pi_n$ that leads to an anti-isomorphism between $\mbf{\Pi}_n/{\sim}$ and some $[f]{\downarrow}$.
Restricting the minor relation to $\homi_{\alg{A}\alg{B}}$ gives a much simpler situation, as shown in the next result.

\begin{prop}\label{partition}
For every $f\in \homi_{\textbf{A}\textbf{B}}$ the principal ideal $[f]{\downarrow}$ in $\mbf{\homi}_{\alg{A}\alg{B}}$ is anti-isomorphic to the partition lattice $\alg{\Pi_{\ess(f)}}$.        
\end{prop}

\begin{proof}
We can assume that $f\in\homi_{\alg{A}\alg{B}}^{(n)}$, where $n$ is the essential arity of $f$. We know by Corollary \ref{cor:iso1} that $[f]{\downarrow}$ is order-isomomorpic to $[f^*]_d{\downarrow}$, and we prove that  $[f^*]_d{\downarrow}$ is anti-isomorphic to $\alg{\Pi_{\ess(f)}}$. We use the notation defined in (the paragraph preceding) Definition \ref{defn:fpi}.

Let $\phi\colon \Pi_n \rightarrow [f^*]_d{\downarrow}$ be the map defined by $\phi(\pi) = [f_\pi^*]_d$.   
%Any element $\psi\colon\alg{B}^* \to |\pi|\alg{A}^*$ of $[f_\pi^*]_d$ can be thought of as 'overlapping' the copies of $A^*{\setminus}\mathcal{C}$ whenever they belong to the same block of $\pi$. 
We have already noted that $\phi$ is an onto, order-reversing map, and now we prove that $\phi$ is one-to-one.  Let $\pi_1$ and $\pi_2$ be distinct elements of $\Pi_n$ and for $i\in \{1,2\}$ let 
\begin{equation}\label{eqn:bfr}
\psi_i:= (\hat{c}^i)_\alg{X} \circ f^*
\end{equation}
where $c^i\colon \pi_i \to [|\pi_i|]$ is an arbitrary bijective labeling of the elements of $\pi_i$. We show that $\psi_1\not \equiv_d \psi_2$. By symmetry, we may assume that there is some block $C\in \pi_1 {\setminus} \pi_2$, and we let $\ell$ be an element of $C$ and $D := \pi_2^\ell$ be the unique block of $\pi_2$ containing $\ell$. Since $C \neq D$, we may assume that there is some $k\in C{\setminus} D$ (the case $k\in D{\setminus} C$ is similar). Since $f$ has no inessential arguments, we know by Lemma \ref{lem:essvar} that the $\ell$-th and $k$-th co-arguments of $f^*$ are essential. This means that there are some $u,v\in \alg{B}^*$ such that 
\[
f^*(u)\in \{\ell\}\times \alg{A}^{*\flat} \quad \text{and} \quad f^*(v)\in \{k\}\times \alg{A}^{*\flat}.
\]
By construction, we obtain 
\[
\{\psi_1(u), \psi_1(v)\}\subseteq \{c_1(C)\}\times \alg{A}^{*\flat}, \quad \psi_2(u)\in \{c_2(D)\}\times \alg{A}^{*\flat}, \quad  \psi_2(v)\notin \{c_2(D)\}\times \alg{A}^{*\flat},
\]
which shows that $\psi_1$ maps $u$ and $v$ into the same copy of $\alg{A}^{*\flat}$ in $|\pi_1|\alg{A}^{*}$, while $\psi_2$ maps $u$ and $v$ into two different copies of  $\alg{A}^{*\flat}$ in $|\pi_2|\alg{A}^*$. We conclude that $\psi_1 \not\equiv_d \psi_2$.
%Let $\pi_1$ and $\pi_2$ be distinct elements of $\Pi_n$, and $\psi_i$ be an element of $[f^*_{\pi_i}]_d$ for $i\in \{1,2\}$. We show that $\psi_1\not \equiv_d \psi_2$ . By symmetry, we may assume there is some block $C\in \pi_1 {\setminus} \pi_2$, and we let $i$ be an element of $C$ and $D$ be the unique block of $\pi_2$ that contains $i$. Since $C \neq D$, we may assume that there is some $j\in C{\setminus} D$ (the case $j\in D{\setminus} C$ is similar). Now, the sets 
%\begin{equation*}
%X_i:=(f^*)^{-1}(\{i\}\times \alg{A}^{*\flat})\quad \text{ and } \quad X_j:= (f^*)^{-1}(\{j\}\times \alg{A}^{*\flat}),
%\end{equation*} 
%are both nonempty since $f^*$ has no inessential arguments.  
%Since $i$ and $j$ belong to the same block $C$ of $\pi_1$ but to distinct blocks of $\pi_2$, we have that $\psi_1(X_i)$ and $\psi_1(X_j)$ are subsets of the \emph{same} copy of $\alg{A}^{*\flat}$ in $|\pi|\alg{A}^*$, while $\psi_1(X_i)$ and $\psi_1(X_j)$ are in \emph{distinct} copies of $\alg{A}^{*\flat}$ in $|\pi|\alg{A}^*$. It follows that $\psi_1 \not\equiv_d \psi_2$, that is, $\varphi(\pi_1)\neq \varphi(\pi_2)$.

It remains to show that $\varphi^{-1}$ is order-reversing. Let $\pi_1, \pi_2\in \Pi_n$ such that $\pi_1 \not\geq \pi_2$ and show that $\varphi(\pi_1) \not\leq_d \varphi(\pi_2)$. Let $C$ be a block of $\pi_2$ which is not contained in any block of $\pi_1$. There are two distinct elements $i,j\in C$ which belong to two distinct blocks of $\pi_1$. For every $i\in\{1,2\}$, denote by $c_i$ a bijective labeling $c_i\colon \pi_i \to [|\pi_i|]$, and let $\psi_i \in [f_{\pi_i}]$ be defined as in \eqref{eqn:bfr}. By a similar argument as in the first part of the proof, we can find  $u,v\in \alg{B}^*$ such that $\psi_2$ maps $u$ and $v$ into the same copy of $\alg{A}^{*\flat}$ in $|\pi_2|\alg{A}^*$, while $\psi_1$ maps $u$ and $v$ into two distinct copies of $\alg{A}^{*\flat}$ in $|\pi_1|\alg{A}^*$. This shows that $\psi_1\not \preceq_d \psi_2$, and therefore $\varphi(\pi_1)\not \leq_d \varphi(\pi_2)$ as desired.  
\end{proof}

Recall that the \emph{arity gap} of $f\in \func_{AB}$ is defined as the minimum difference between that the essential arity of $f$ and of that of an identification minor of $f$. We retrieve the following result, which is a special instance of {\cite[Proposition 3.13]{EssVar}, as a consequence of Proposition \ref{partition}.

\begin{cor}\label{cor:tidim}
For every $\alg{A}, \alg{B} \in \var{A}$ and $n>1$, the arity gap of $f\in \homi_{\alg{A}\alg{B}}^{(n)}$ is one.
\end{cor}

Functions $f\in \func_{{A}{B}}$ that have a unique identification minor have received special interest and have been studied in relation with their invariance group (see \cite{Lehtonen2016, Lehtonen2014}). Recall that the invariance group $\Inv(f)$ of $f\in \func_{{A}{B}}^{(n)}$ is defined to be the subgroup of $\alg{S}_n$ given by $\{\sigma\in \alg{S}_n \mid f=f\circ \sigma^A\}$. If $f\in\homi_{\alg{A}\alg{B}}^{(n)}$, we obtain the following result.

\begin{prop}\label{prop:tadam}
Let $f\in \homi_{\alg{A}\alg{B}}^{(n)}$ with $\ess(f)=n$.
\begin{enumerate}
\item\label{it:rop01} $f$ has $\binom{n}{2}$ pairwise non-equivalent identification minors.
\item\label{it:rop02} $\Inv(f)$ is trivial.
\end{enumerate}
\end{prop}
\begin{proof}
(\ref{it:rop01}) is a direct consequence of Proposition \ref{partition}.

(\ref{it:rop02}) An element $\sigma\in \alg{S}_n$ belongs to $\Inv(f)$ if and only if the diagram
\begin{equation}\label{eqn:cominv}
\begin{tikzcd}[column sep=small]& A^n\arrow[dl, "\sigma^A"']\arrow[dr,"f"] & \\A^n\arrow[rr, "f"] && B\end{tikzcd}
\end{equation}
commutes. Diagram \eqref{eqn:cominv} is, by Assumption \ref{ass:main}, equivalent  to
\begin{equation}
\begin{tikzcd}[column sep=small]& nA^*\arrow[dl, leftarrow, "\sigma_{\alg{A}^*}"']\arrow[dr,leftarrow, "f^*"] & \\nA^*\arrow[rr, leftarrow, "f^*"] && B^*\end{tikzcd}
\end{equation}
which commutes if and only if $\sigma=\id_n$ since $\ess(f)=n$.
\end{proof}

In the following example, we show how Proposition \ref{partition} may fail if we weaken Assumption \ref{ass:main} by assuming that $\utilde{\alg{M}}$ yields a duality which is not logarithmic.

\begin{exam} Let $\var{BG}$ be the variety of Boolean groups and $\underline{\alg{Z}_2}=\struc{\{0,1\}, +, 0}$ as in \ref{ex:boolgroups}. The homomorphism $f\in\var{BG}_{\underline{\alg{Z}}}^{(3)}$ given by $f(x,y,z)=x+y+z$ has essential arity $3$. But $f$ has an unique identification minor (namely, the identity map) and arity gap $2$. 
\end{exam}

Proposition \ref{partition} shows that the deck of any $f\in \homi_{\alg{A}\alg{B}}$  could not actually be richer (it does not contain any duplicates). This observation leads us to the investigation of reconstructibility properties for homomorphisms, as  introduced in Section \ref{Minor}. 

As explained in  \cite{Reconstruction}, the class $\homi_{\alg{A}\alg{B}}$ cannot be reconstructible. Indeed, let $\alg{A}, \alg{B}\in \var{A}$ and $n<|A|$. Let $A_{\neq}^n \subseteq A^n$ consists of all elements of $A^n$ which are injective on $[n]$ (i.e. the tuples in which no entry appears more than once). If $f\in\homi_{\alg{A}\alg{B}}$ has arity $n$, then any map $g\colon A^n \to B$ which is equal to $f$ on $A^n{\setminus} A^n_{\neq}$ satisfies $\deck(f)=\deck(g)$, which shows that $f$ is not reconstructible. However, such a map $g$ is unlikely to still belong to $\homi_{\alg{A}\alg{B}}$, which naturally leads to the weak reconstruction problem for $\homi_{\alg{A}\alg{B}}$.

%{\color{red}What can be said for the reconstructibility problem for $\homi_{\alg{A}\alg{B}}^{>n}$ for $n\geq |A|$? }

\begin{thm}\label{thm:weakly}
Let $\alg{A}, \alg{B}\in \var{A}$. The subclass $\homi_{\textbf{A}\textbf{B}}^{>2}$ of homomorphisms in $\homi_{\alg{A}\alg{B}}$ of essential arity strictly greater than $2$ is weakly reconstructible.  
\end{thm} 

\begin{proof}
We proceed by contradiction, assuming that there are $f,g\in \homi_{\alg{A}\alg{B}}^{(n)}$ for some $n>2$ with  $\deck(f)=\deck(g)$ but $f\not\equiv g$. We may assume that $f$ and $g$ both have no inessential arguments. Due to Assumption \ref{ass:main} and Lemma \ref{lem:essvar}, this means that  every co-argument of $f^*$ and $g^*$ is essential and $f^*\not\equiv_d g^*$.

%First, we show that if $u\in\alg{B}^*$ satisfies $f^*(u)\not\in \mathcal{C}^{n\alg{A}^*}$ then $g^*(u)\not\in \mathcal{C}^{n\alg{A}^*}$ and $\pr_2(f^*(u))=\pr_2(g^*(u))$.

First we show $(f^*)^{-1}(\mathcal{C}^{n\alg{A}^*}) = (g^*)^{-1}(\mathcal{C}^{n\alg{A}^*})$ and for all $u\in \alg{B}^* {\setminus}(f^*)^{-1}(\mathcal{C}^{n\alg{A}^*})$ we have that $\pr_2(f^*(u)) = \pr_2(g^*(u))$. For the sake of contradiction, assume that there is some element $u\in \alg{B}^*$ with $f^* (u)\in\mathcal{C}^{n\alg{A}^*}$ and $g^*(u)\notin\mathcal{C}^{n\alg{A}^*}$ (the case $f^*(u),g^*(u)\not\in \mathcal{C}^{n\alg{A}^*}$ and $\pr_2(f^*(u)) \neq \pr_2(g^*(u))$ is dealt with similarly). 
%In particular, we have $h^*(u)\notin\mathcal{C}^{n\alg{A}^*}$ for every $h$ such that $h\preceq g$. 
Let $I\in \binom{n}{2}$. We have 
\[(f\circ \delta_I^A)^*(u)=((\delta_I)_{\alg{A}^*}\circ f^*)(u) \text{ belongs to } \mathcal{C}^{(n-1)\alg{A}^*}\]
 while 
 \[
 (g\circ \delta_J^A)^*(u)=((\delta_J)_{\alg{A}^*}\circ g^*)(u) \text{ does not belong to } \mathcal{C}^{(n-1)\alg{A}^*},
 \] 
 for every $J\in \binom{n}{2}$. It follows that the equivalence class of $f\circ \delta_I^A$ belongs to $\deck(f)$ but not to $\deck(g)$, a contradiction.
 
 Then, since $f^*\not\equiv_d g^*$ there is some $u\in\alg{B}^*$ such that $\pr_1(f^*(u))\neq \pr_1(g^*(u))$. For every $i\in [n]$, set
\[
X_i:=(f^*)^{-1}(\{i\} \times \alg{A}^{*\flat}) \quad \text{ and } \quad Y_i:=(g^*)^{-1}(\{i\}\times \alg{A}^{*\flat}).
\] 
We have proved that 
\begin{equation}\label{eqn:partition}
X:=\{X_i \mid i \in [n]\}\quad \text{ and }\quad Y:=\{Y_i \mid i \in [n]\}
\end{equation}
 are distinct partitions of $\alg{B}^*{\setminus} (f^{*})^{-1}(\mathcal{C}^{n\alg{A}^*})$. Without loss of generality, we may assume that there is some $i\in [n]$ with $X_i\not\in Y$. Let $u$ be an element of $X_i$, and let $Y_j$ be the block of $Y$ that contains $u$. We may assume that $X_i \not\subseteq Y_j$  (the case $Y_j \not\subseteq X_i$ is dealt with
 similarly) and we let $v$ be an element of $X_i{\setminus} Y_j$, and $Y_k$ be the block of $Y$ that contains $v$. Since $n>2$, either $i\neq j$ or $i\neq k$. If $i\neq j$ (the case $i\neq k$ can be dealt with similarly), then $(g\circ \delta_{\{i,k\}}^A)^*(u)$ and $(g\circ \delta_{\{i,k\}}^A)^*(v)$ belong to the $\delta_{\{i,k\}}(j)$-th copy and $\delta_{\{i,k\}}(k)$-th copy  of $\alg{A}^{*\flat}$ in $(n-1)\alg{A}^*$, respectively, and $\delta_{\{i,k\}}(j)\neq\delta_{\{i,k\}}(k)$. On the other hand, for every $I\in \binom{n}{2}$, the map $(f\circ \delta_{I}^A)^*$ maps $u$ and $v$ to the $\delta_I(i)$-th copy of $\alg{A}^{*\flat}$ in $(n-1)\alg{A}^*$. Therefore the equivalence class of $g\circ \delta_{\{i,k\}}^A$ belongs to $\deck(g)$ but not to $\deck(f)$, a contradiction.
\end{proof}

Theorem \ref{thm:weakly} also shows that $\homi_{\alg{A}\alg{B}}^{>2}$ is not recognizable (see Definition \ref{defn:recons}) if it is nonempty and $|A|> 2$. Indeed, we already noted that a homomorphism $f\colon \alg{A}^n \rightarrow \alg{B}$ with $n\leq |A|$ is not reconstructible, while now we showed that all its reconstructions which are homomorphisms are equivalent. Therefore, there needs to be a reconstruction of $f$ which is not a homomorphism.  
\par{\smallskip}  

The first part of the proof of Theorem 4.6 shows that if two $\cate{X}$-morphisms $\phi, \psi \colon \alg{Y} \to n \alg{X}$ have a common minor, then $\phi$ and $\psi$ might map an element $x\in Y$ into two different copies of $\alg{X}$ in $n\alg{X}$, but never to different elements. We generalize this proof in the following lemma.       

\begin{lem}\label{CommonMinor}
Let $\varphi , \psi\colon \alg{Y} \rightarrow n\alg{X}$ be two morphisms such that $[\varphi]_d{\downarrow} \cap [\psi]_d{\downarrow} \neq \varnothing$. Then $\varphi^{-1}(\mathcal{C}^{n\alg{X}}) = \psi^{-1}(\mathcal{C}^{n\alg{X}}) =: P$ and for all $y \in Y{\setminus} P$ we have $\pr_2(\varphi(y)) = \pr_2(\psi(y))$.   
\end{lem}  

\begin{proof}
Let $\mu\colon \alg{Y} \to m\alg{X}$ be a morphism such that $\mu \preceq_d \varphi$ and $\mu \preceq_d \psi$.  By Proposition \ref{partition}, there are $\tau, \tau'\colon [n]\to [m]$ such that $\mu=\tau_\alg{X}\circ \varphi=\tau'_\alg{X}\circ \psi$. In particular, for every $y\in \alg{Y}$, we have the following equivalences
\[
\varphi(y)\in \mathcal{C}^{n\alg{X}} \iff \mu(y)\in \mathcal{C}^{m\alg{X}}\iff  \psi(y)\in  \mathcal{C}^{n\alg{X}}.
\]
Furthermore, for $y\in Y{\setminus} P$ we have 
\begin{align*}
\varphi(y) = (i, x) & \iff \mu(y) = (\tau(i), x) \\
\psi(y)= (j, x)& \iff \mu(y) = (\tau'(j), x),
\end{align*}
from which we get the second part of the statement.
\end{proof} 

\section{Characterizing the minor homomorphism poset}\label{MinorHomomorphism}
Proposition \ref{partition} and Lemma \ref{CommonMinor} pave the way to a complete description of the posets $\mbf{\homi_{\alg{A}\alg{B}}}$ in terms of partition lattices. Prior to developing the general tools to characterize these posets, we look at a few examples to get a better understanding of how they are influenced by the structure defined by $\utilde{\alg{M}}$. We use the vocabulary and notation introduced in Section \ref{Examples}, in particular, $\mathcal{B}$ and $\mathcal{D}$ denote the variety of Boolean algebras and unbounded distributive lattices, respectively.

\begin{exam}\label{ex:boolfin}
\emph{The minor homomorphism poset $\mbf{\var{B}}_{\alg{2}^k}$ of the finite Boolean algebra $\alg{2}^k$ consists of $k^k$ disjoint copies of the order dual of the $k$-th partition lattice:}
\[\mbf{\var{B}}_{\alg{2^k}}\simeq \biguplus_{1\leq i\leq k^k} \alg{\Pi}_k^\partial.\]
\begin{proof}
Recall that $[k] = \{ 1,\dots, k \}$ is the dual of $\alg{2}^k$ under Stone duality. We first note that a homomorphism $f\colon (\alg{2}^k)^n \rightarrow \alg{2^k}$ satisfies $\ess(f)\leq k$, since $\im(f^*)$ can meet at most $k$ copies of $(\alg{2^k})^*$ in $n(\alg{2^k})^*$. So, maximal elements $[\varphi]_d$ of $\mbf{\cate{X}}_{(\alg{2^k})^*}$ are represented by maps $\varphi\colon [k] \rightarrow k[k]$ with $\ess_d(\varphi) = k$. If $\varphi$ and $\psi$ are two such maps, then $\varphi \equiv_d \psi$ if and only if $\pr_2 \circ \varphi =\pr_2\circ \psi$. Thus, there are $k^k$  maximal classes in $\mbf{\cate{X}}_{(\alg{2^k})^*}$. For each maximal class $[\varphi]_d$, the corresponding class $[\varphi_*]$ is maximal in $\mbf{\var{B}}_{\alg{2^k}}$. We know by Proposition~\ref{partition} that $[\varphi_*]{\downarrow} \simeq \alg{\Pi}_k^\partial$. Moreover, for two maximal elements $[\varphi]_d \neq [\psi]_d$, we have $\pr_2\circ\varphi \neq \pr_2\circ\psi$, so $[\varphi]_d{\downarrow} \cap [\psi]_d{\downarrow} =\varnothing$ holds by Lemma \ref{CommonMinor}.
\end{proof} 
\end{exam}

\begin{rem}
By extending the argument of Example \ref{ex:boolfin},  one can show that  
\[ 
\mbf{\var{B}}_{\alg{2^\ell},\alg{2^k}} \simeq \biguplus_{1\leq i\leq \ell^k} \Pi_k^\partial.
\]
The only difference here is that the morphisms $\varphi\colon [k] \rightarrow k[\ell]$ which represent maximal elements are now characterized by all the distinct elements of  $[\ell]^{[k]}$. 
\end{rem}

\begin{exam}\label{ex:lat}\emph{The minor homomorphism poset of the distributive lattice $\alg{L}$ depicted in  Fig.~\ref{fig:lat} is isomorphic the disjoint union of $30$ copies of the two-element chain and an antichain of order $35$:}
\begin{equation*}
\mbf{\var{D}}_{\alg{L}} \simeq \biguplus_{1\leq i \leq 30} \alg{\Pi}_2^\partial \uplus \biguplus_{1\leq j\leq 29} \alg{\Pi}_1^\partial \uplus \biguplus_{1\leq k \leq 6} \alg{\Pi}_{0}^\partial.
\end{equation*}
\begin{proof}
The dual $\alg{L}^*$ of $\alg{L}$ is computed as described in Subsection \ref{sub:dl} as the finite bounded poset $\struc{J(\alg{L})_{01}, \leq_\alg{L}, 0,1}$ of the join-irreducible elements $J(\alg{D})=\{a, b, c\}$ of $\alg{L}$ with constants $0$ and $1$ added, and is depicted in Fig.~\ref{fig:lat}.
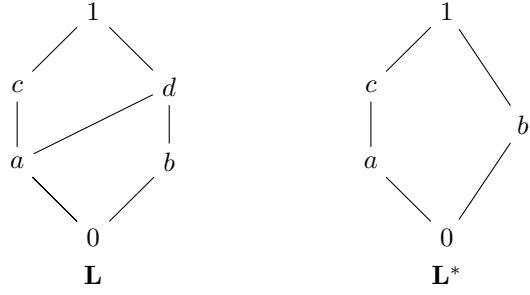
\begin{figure}
\begin{tikzpicture}
  \node (0) at (0,0) {$0$};
  \node (1) at (0,3) {$1$};
  \node (p) at (-1,1) {$a$};
  \node (q) at (1,1) {$b$};
  \node (r) at (-1,2) {$c$};
  \node (s) at (1,2) {$d$};
  \draw (0) -- (p) -- (r) -- (1) -- (s) -- (q) -- (0) -- (p) -- (s);
  \draw (0,-0.5) node{$\alg{L}$};
\end{tikzpicture}
\hspace{2cm}
\begin{tikzpicture}
  \node (0) at (0,0) {$0$};
  \node (1) at (0,3) {$1$};
  \node (r) at (-1,1) {$a$};
  \node (p) at (-1,2) {$c$};
  \node (q) at (1,1.5) {$b$};
  \draw (0) -- (r) -- (p) -- (1) -- (q) -- (0);
  \draw (0,-0.5) node{$\alg{L}^*$};
\end{tikzpicture}\caption{Distributive lattice $\alg{L}$ of Example \ref{ex:lat} and its dual $\alg{L}^*$.}\label{fig:lat}
\end{figure}

A morphism $\varphi\colon \alg{L}^* \rightarrow n\alg{L}^*$ needs to be order-preserving and to preserve $0$ and $1$. Therefore, if $\{\varphi(a), \varphi(c)\} \cap \{0,1\}=\varnothing$, then there is one index $i\in[n]$ such that $\{\varphi(a), \varphi(c)\}  \subseteq \{i\}\times \alg{L}^{*\flat}$. In particular, we have $\ess_d(\varphi)\leq 2$, since $\varphi(\alg{L}^*)$ can meet at most 2 different copies of $\alg{L}^{*\flat}$ in $n\alg{L}^*$ ($\varphi(b)$ may not be in the same copy of $\alg{L}^{*\flat}$ as $\{\varphi(a),\varphi(c)\}$). 

Hence,  up to equivalence $\equiv_d$, there are $30$ morphisms $\varphi\colon \alg{L}^* \to 2\alg{L}^*$ of co-essential arity $2$ and they are defined by the constraints
\begin{multline}
\big(\pr_2(\varphi(a)),\pr_2(\varphi(c))\big)\in \big\{(0,a), (0,b), (0,c), (a,c), (a,1),    (c,1), \\ (b,1), (a,a), (b,b), (c,c) \big\},\label{eqn:bvd0}
\end{multline}
\begin{gather}
\pr_2(\varphi(b))\in \{a,b,c\}, \label{eqn:bvd6}\\
\{\pr_1(\varphi(a)), \pr_1(\varphi(c))\}=\{1\}, \quad 
\pr_1(\varphi(b))=2.
\end{gather} 
They correspond to $30$ maximal elements in $\mbf{\var{D}}_{\alg{L}}$. 

Moreover, there are $29$ maximal elements $[\varphi]_d$ where $\varphi\colon\alg{L}^* \to  \alg{L}^*$ has co-essential arity 1. These maps satisfy either $\varphi(b)\in\{0,1\}$ and \eqref{eqn:bvd0}, or $\{\varphi(a),\varphi(c)\}\subseteq \{0,1\}$ and (\ref{eqn:bvd6}).

To complete the picture, there are $6$ non-equivalent nullary morphisms $\varphi$ defined by the constraints $\varphi(\alg{L}^*)\subseteq \mathcal{C}$ and $\varphi(a)\leq \varphi(c)$, and they also all correspond to maximal elements in $\struc{\homi_{\alg{L}}/{\equiv}, \leq}$.
\end{proof} 
\end{exam}

Example \ref{ex:boolfin} illustrates that a minor homomorphism poset can be much smaller than the corresponding minor poset. For instance, while $\mbf{\var{B}}_{\alg{2^k}}$ is of finite size $k^k  B_k$, the poset $\mbf{\mathcal{F}}_{\{0,1\}}$ is countably infinite and contains a copy of every finite poset, as shown in \cite{QuasiBool}.

Nevertheless, even for Boolean algebras $\alg{A}\in\var{B}$ the minor homomorphism poset $\mbf{\var{B}}_{\alg{A}}$ can get quite complex in the infinite case. There are two reasons why things get more complicated in this case: the essential arities might be unbounded and the topology comes into play. The following example illustrates this phenomenon (see also Example \ref{ex:bool2}).

\begin{exam}\label{ex:fini-cofini}\emph{The minor homomorphism poset of the Boolean algebra $\alg{A}$ of finite and cofinite subsets of $\N$ contains countably infinite chains, and uncountably infinite antichains.}
\begin{proof}The dual space $\textbf{A}^*$ is given by the one-point compactification $\N\cup\{ \omega \}$ of $\N$ (or, equivalently, the ordinal $\omega + 1$ with the order topology). Explicitly this means that for any $U\subseteq \alg{A}^*$ we have
\begin{equation*}
U\text{ is open }\quad \iff \quad U \subseteq \N \quad \text{ or }\quad (\omega \in U \text{ and } U\cap \N \text{ is cofinite}).  
\end{equation*}
The dual of a homomorphism $f\colon \textbf{A}^n \rightarrow \textbf{A}$ is a  continuous map $f^* \colon \alg{A}^* \rightarrow n\alg{A}^*$ and vice versa. Dealing with the topology on $\alg{A}^*$ easily leads to the construction of infinite chains and infinite antichains in $\mbf{\var{B}}_{\alg{A}}$, for instance. 

To construct an uncountable antichain, for every $S\subseteq\N$ let $c_S\colon \alg{A}^* \to  \alg{A}^*$ be the continuous map defined by
\begin{equation*}
c_S(x) = \begin{cases}
     \omega & \text{if } x \in S \\
      x & \text{otherwise.} 
   \end{cases}
\end{equation*}
Then $\{[{c_S}_*] \mid S\subseteq \N\}$ is an uncountable antichain in $\mbf{\var{B}}_{\alg{A}}$.

Now we construct a countable chain above an element $[\varphi]_d$ where $\varphi$ is any unary morphism.    
For every $n\geq 2$, we define $\varphi_n \colon \alg{A}^* \rightarrow n\alg{A}^*$ by
\begin{equation*}
\varphi_n(u) = \begin{cases}
     (1, \varphi(u)) & \text{if } u \in \alg{A}^*{\setminus} \{ 1,\dots,n - 1 \}  \\
     (u+ 1, \varphi(u)) & \text{if } u\in \{ 1,\dots,n - 1 \}. 
   \end{cases}
\end{equation*}       
For every $n\geq 2$, the map $\varphi_n$ is continuous (we have $\varphi_n(u)=\varphi(u)$ for all but a finite number of $u\in \alg{A}^*$) with $\ess_d(\varphi_n)=n$ and $\varphi_n \prec_d \varphi_{n+1}$,  since $\varphi_{n}=(\delta_{I})_{\alg{A}^*} \circ \varphi_{n+1}$ holds for $I=\{1,n+1\}$. Hence, the map $n\mapsto [\varphi_n]_d$ embeds $\struc{\N, \leq}$ in $\mbf{\cate{X}}_{\alg{A}^*}$ for all $\varphi \in \cate{X}_{\alg{A}^*}^{(1)}$. 
\end{proof}
\end{exam}

In Examples \ref{ex:boolfin} and \ref{ex:lat} the corresponding sets of essential arities $\ess(\homi_{\alg{A}\alg{B}}) := \{ \ess(f) \mid f\in \homi_{\alg{A}\alg{B}}\}$ have an upper bound. We generalize these examples in the following result.

\begin{thm}\label{MaxDisjUnion} 
Let $\textbf{A},\textbf{B}\in\mathcal{A}$ be algebras such that $\ess(\homi_{\alg{A}\alg{B}})$ is bounded and let $K$ be the collection of maximal elements in $(\homi_{\alg{A}\alg{B}}/{\equiv},\leq)$. Then 
\begin{equation*}
\mbf{\homi}_{\alg{A}\alg{B}} \simeq \biguplus_{[f]\in K} \alg{\Pi}^\partial_{\ess(f)}.
\end{equation*} 
\end{thm}

\begin{proof}
Since $\ess(\homi_{\alg{A}\alg{B}})$ is bounded, for every homomorphism $h \in \homi_{\alg{A}\alg{B}}$ there is a maximal element $m_h \succeq h$, and we need to show that $m_h$ is unique. For the sake of contradiction, assume that there are maximal elements $[f]$ and $[g]$ of $\mbf{\homi}_{\alg{A}\alg{B}}$  with $[f^*]_d{\downarrow} \cap [g^*]_d{\downarrow}\neq\varnothing$. We can assume that $f$ and $g$ both have no inessential arguments. By Lemma \ref{CommonMinor} the maps $f^*$ and $g^*$ differ only in how they distribute elements of $\alg{B}^*$ to different copies of $\alg{A}^{*\flat}$, and we  construct the partitions $X$ and $Y$ of $\alg{B^*}{\setminus} f^{*-1}(\mathcal{C}^{n\alg{A}^*})$ as we did in \eqref{eqn:partition}.

Since $[f^*]_d$ is maximal and  $[g^* ]_d\not\leq_d [f^*]_d$, we can find a block $X_i\in X$ and a block $Y_j\in Y$  such that both $X_i \cap Y_j$ and $X_i{\setminus} Y_j$ are nonempty. Now consider the map $\varphi \colon \alg{B}^* \rightarrow (n+1)\alg{A}^*$ defined by
\begin{equation*}
\varphi (u) = \begin{cases}
     \big(n+1, \pr_2( f^* (u))\big) & \text{if } u \in X_i\cap Y_j  \\
     f^*(u) & \text{otherwise. } 
   \end{cases}
\end{equation*}
We prove that $\varphi$ is a morphism and that $\ess_d(\varphi)=n+1$.
%This map $\phi$ has essential arity $n+1$ since both $F_i \cap G_j$ and $F_i{\setminus} G_j$ are non-empty. Furthermore, it is still a morphism:
For the sake of contradiction,  assume that there is some $k$-ary relation $r$ such that $(u_1,\ldots, u_k)\in r^{\alg{A}^*}$  but $(\varphi(u_1),\ldots, \varphi(u_k)) \not\in r^{n\alg{A}^*}$ in the type of $\alg{A}^*$ (the argument for partial or total functions is similar). Since $f^*$  preserves relations, we have $X_i \cap Y_j \cap \{ u_1,\ldots, u_k \}\neq \varnothing$ and $(X_i{\setminus} Y_j) \cap \{u_1\dots, u_k\}\neq \varnothing$. In particular, $g^*(\{u_1,\dots, u_k\})$ meets at least two different copies of $\alg{A}^{*\flat}$ in $n\alg{A}^*$ and cannot preserve $r$, a contradiction. 

Now, we prove that $\varphi$ is continuous. Let $k\in [n+1]$ and $C$ be a closed subset of $\{ k \} \times \alg{A}^{*\flat}$. Then
\begin{equation*}
\varphi^{-1}(C) = \begin{cases}
     (f^*)^{-1}(C) \cap X_i \cap Y_j & \text{ if } k = n+1  \\ 
     (f^*)^{-1}(C){\setminus} (X_i \cap Y_j) & \text{ otherwise }    \\
   \end{cases}   
\end{equation*}
is  closed since $X_i$ and $Y_j$ are clopen and $f^*$ is continuous.   Moreover, we have $\ess_d(\varphi)=n+1$ by definition of $\varphi$ since $f$ has essential arity $n$.

Now, we clearly have $[f^*]_d<[\varphi]_d$ by construction, which contradicts the maximality of $[f]$.
\end{proof}

Theorem \ref{MaxDisjUnion}  shows that $\mbf{\homi}_{\alg{A}\alg{B}}$ is completely determined by the essential arities of its maximal elements. Now, we investigate how to recognize these maximal elements in the dual category. Informally, they are represented by those morphisms $\varphi\colon \alg{B}^* \to n\alg{A}^*$ without inessential co-arguments such that $\varphi(\alg{B})^* \cap \big(\{i\}\times \alg{A}^*\big)$ cannot be decomposed into two substructures that would lead to a morphism covering $\varphi$ in $\struc{\cate{X}_{\alg{B}^*,\alg{A}^*}, \preceq}$.

\begin{defin}
A substructure $\alg{Y}$ of $\alg{X}\in \cate{X}$ is called \emph{complete} if it satisfies the following conditions.
\begin{enumerate}
\item\label{it:mlk} For every $k$-ary relation $r$ of $\alg{X}$ and every $x_1, \ldots, x_k \in X$
\[
\big(\{x_1, \ldots x_k\} \cap Y \neq \varnothing \text{ and } (x_1\ldots, x_k) \in r\big)\implies \{x_1, \ldots, x_k\}\subseteq Y.
\]
\item For every partial or total operation $O$ of $X$, the graph $r_O$ of $O$ satisfies condition (\ref{it:mlk}). 
%\item For every $k$-ary partial operation $\tau$ of $X$ an every $x_1, \ldots, x_k\in X$,
%\[
%\big((x_1, \ldots, x_k)\in \dom(\tau) \text{ and }\{x_1, \ldots, x_k\} \cap Y \neq \varnothing\big) \implies \{x_1, \ldots, x_k\} \subseteq Y.
%\]
\end{enumerate}
For every $S\subseteq \alg{X}$ we denote by $\struc{S}$ the smallest complete substructure of $\alg{X}$ that contains $S$.
For every $\alg{X}\in\cate{X}$, let $\sim$ denote the equivalence relation defined on $\alg{X}^\flat$ by
$
x \sim y  \text{ iff }  \struc{\{x\}}=\struc{\{y\}}.
$
\end{defin}
\begin{lem}\label{lem:awf}
For every $\alg{X}\in \cate{X}$ and every $x,y\in \alg{X}^\flat$ we have $x\sim y$ if and only if there are relations $r_1, \ldots, r_\ell \in \mathcal{R}^{\alg{X}}\cup \{r_O \mid O \in \mathcal{G}\cup \mathcal{H}\}$ of arity $k_1, \ldots, k_\ell$, respectively, and $\mathbf{x}^i \in r_i$ for all $i\leq \ell$ such that $x\in\{x^1_1, \ldots x^1_{k_1}\}$, $y\in \{x^\ell_1, \ldots, x^\ell_{k_\ell}\}$, and $\{x^j_1, \ldots, x^{j}_{k_j}\}\cap\{x^{j+1}_1, \ldots, x^{j+1}_{k_j}\}\neq \varnothing$ for all $j\leq \ell-1$. 
\end{lem}

From the perspective of topology, there is no reason for $\sim$ to have nice properties. For instance, we cannot assume  that  the classes of $\sim$ are closed or open, although this is obviously true if $\alg{X}$ is finite. 
\begin{defin}
We say that $\alg{X}\in \cate{X}$ has the \emph{FCO property} if $\alg{X}^\flat/{\sim}$ is finite and if $E \cup \mathcal{C}^\alg{X}$ is a clopen subspace of $\alg{X}$ for every equivalence class $E\in \alg{X}^\flat/{\sim}$.
\end{defin}

\begin{prop}\label{prop:bbound}
Let $\alg{A}, \alg{B}\in \var{A}$. If $\alg{B}^{*}$ has the FCO property, then 
$\ess(\homi_{\alg{A}\alg{B}})\leq |\alg{B}^{*\flat}/{\sim}|$. If in addition $\alg{A} =\alg{B}$, then  $\ess(\homi_\alg{A}) = |\alg{A}^{*\flat}/{\sim}|$. 
\end{prop} 

\begin{proof}
Let $\varphi\colon\alg{B}^* \to n\alg{A}^*$ be a morphism, and $u,v\in \alg{B}^{*\flat}$. By Lemma \ref{lem:awf}, if $u\sim v$ and $\{\varphi(u), \varphi(v)\}\subseteq (n\alg{A}^*)^\flat$, then $\varphi(u)$ and $\varphi(v)$ need to belong to the same copy of $\alg{A}^{*\flat}$ in $n\alg{A}^*$, which proves the first assertion.

For the second part of the statement, let $E_1, \ldots, E_n$ be the elements of $\alg{A}^{*\flat}/{\sim}$.  Define $\psi\colon \alg{A}^* \rightarrow n\alg{A}^*$ as the map that preserves constants and satisfies  $\psi(u)=(i,u)$ for any $i\leq n$ and any $u\in E_i$. Then, the map $\psi$ is a co-essentially $n$-ary morphism.        
\end{proof} 

Structures with the FCO property can sometimes be constructed using finite products of algebras, as illustrated in the next result.

\begin{cor}\label{cor:ess}
If $\alg{A}_1, \ldots, \alg{A}_k\in \var{A}$ satisfy $|(\alg{A}_i)^{*\flat}/{\sim}|=1$ for every $i\leq k$, then the dual of $\alg{A}:=\alg{A}_1\times \cdots \times \alg{A}_k$ has the FCO property. Moreover, we have $\ess(\homi_{\alg{B}\alg{A}})\leq k$ for every $\alg{B}\in \var{A}$, and $\ess(\homi_{\alg{A}})=k$.
\end{cor}
\begin{exam}Let $\alg{L}_\omega$ be the distributive lattice whose elements are $\N$ and its finite subsets, ordered by inclusion. Its dual Priestley space $\alg{L}_\omega^*$ is given by an enumerable antichain (made out of the filters generated by singletons $\{i\}$ for $i\in \N$), the elements of which are all covered by some element $\omega$, and two additional bounds $0$ and $1$. It follows from Lemma \ref{lem:awf} that $|\alg{L}^{*\flat}/{\sim}|=1$. By Corollary \ref{cor:ess}, any  $f\in  \var{D}_{\alg{L}\alg{L}_\omega^k}$ has essential arity at most $k$ for every bounded distributive lattice $\alg{L}$ and every $k\geq 1$, in particular there is no $f\in\var{D}_{\alg{L}\alg{L}_\omega}$ with $\ess(f)\geq 2$. 
\end{exam}

We can recover some Arrow type impossibility result from Corollary \ref{cor:ess}. 

\begin{cor}[{\cite [Corollary 4]{Couceiro2016}}]
Let $\alg{C}_1, \ldots, \alg{C}_k, \alg{D}$ be chains. A map $f\colon\alg{C}_1 \times \cdots \times \alg{C}_k \to \alg{D}$ is a median algebra homomorphism if and only if there is an $i\in[k]$ and a monotone map $g\colon\alg{C}_i \to \alg{D}$ such that $f=g\circ \pr_i$.
\end{cor}

The argument of the second part of the proof of Lemma \ref{prop:bbound} can be generalized as a useful Lemma that enables us to construct homomophism majors of elements of $\homi_{\alg{A}\alg{B}}$. 

\begin{lem}\label{Contin}
Assume that $\alg{A}, \alg{B}\in \var{A},$ and that $\alg{B}^*$ has the FCO property.
For any morphism $\varphi \colon \alg{B}^* \rightarrow n\alg{A}^*$  and any $E$ in $\alg{B}^{*\flat}/{\sim}$, the map $\psi \colon \alg{B}^* \rightarrow (n+1)\alg{A}^*$ defined by
\begin{equation*}
\psi(u) = \begin{cases}
    \big(n+1, \pr_2(\varphi(u))\big) & \text{if } u \in E{\setminus}\varphi^{-1}({\mathcal{C}^{n\alg{A}^*}})  \\
     \varphi(u) & \text{otherwise, } 
   \end{cases}   
\end{equation*}  
is a morphism. If, in addition, $E{\setminus}\varphi^{-1}({\mathcal{C}^{n\alg{A}^*}})\neq\emptyset$, then $\varphi \prec_d \psi$.
\end{lem}
\begin{proof}
The map $\psi$ is structure preserving by construction. We need to show that it is also continuous. Let $C$ be a closed subset of $(\{i\} \times \alg{A}^{*\flat})\cup \mathcal{C}^{n\alg{A}^*}$ for some $i\in [n+1]$ and let $\ell$ be the unique element of $[n]$ such that $\varphi(E)$ meets $\{\ell\}\times \alg{A}^{*\flat}$. We have
\[
\psi^{-1}(C)=\begin{cases}
\varphi^{-1}(\{\ell\}\times pr_2(C)) \cap E & \text{if } i = n+1\\
\varphi^{-1}(C) \cap (\alg{B}^*{\setminus} E) &  \text{otherwise,}
\end{cases}
\]
which shows that $\psi$ is continuous by continuity of $\varphi$ and the fact that $E$ is clopen by assumption. Moreover, we have $\varphi=(\delta_I)_{\alg{A}^*}\circ \psi$ where $I = \{ \ell, n+1 \}$.
\end{proof}

We have the following dual characterization of maximal elements in $\mbf{\homi}_{\alg{A}\alg{B}}$.
%We can now describe (all) the morphisms in $\mathcal{X}(\textbf{B}^*, n\textbf{A}^*)$, in particular the representatives of (duals of) maximal elements in the minor homomorphism poset. 
Observe that for any $E\in \alg{B}^{*\flat}/{\sim}$ the subspace $E^\#:=E\cup\mathcal{C}^{\alg{B}^*}$ forms a closed substructure $\mbf{E}^\#$ of $\textbf{B}^*$.

\begin{thm}\label{MorphismsCharacterization}
Assume that $\alg{A}, \alg{B} \in \var{A}$ and that $\alg{B}^*$ has the FCO property, and let  $\varphi \colon \alg{B}^* \to n\alg{A}^*$ be  a map.
\begin{enumerate}
\item\label{itmdvf01} We have $\varphi\in\cate{X}(\alg{B}^*, n\alg{A}^*)$ if and only it for every $E\in \alg{B}^{*\flat}/{\sim}$ with $\varphi(E)\not\subseteq \mathcal{C}^{nA^*}$, there is an $i\in [n]$ such that $\varphi{\restriction}_{\mbf{E}^\#}$ is valued in $\{i\}\times \alg{A}^{*}$ and is a morphism.
\item\label{itmdvf02} If condition (\ref{itmdvf01}) is satisfied, then $[\varphi]_d$ is maximal if and only if for every $E_1\neq E_2$ in $\alg{B}^{*\flat}/{\sim}$ and all $i\leq n$ we have 
\[
\varphi(E_1)\cap (\{i\}\times \alg{A}^{*\flat})\neq \varnothing \quad \implies \quad \varphi(E_2)\cap (\{i\}\times \alg{A}^{*\flat})=\varnothing.
\] 
\end{enumerate}
\end{thm}
\begin{proof}
(\ref{itmdvf01}) The condition is clearly necessary. To prove the converse, it suffices to note that $\varphi$ is structure preserving by definition of $\sim$ and continuous since $E^\#$ is clopen for every $E\in \alg{B}^{*\flat}/{\sim}$.

(\ref{itmdvf02}) To show that the condition is necessary, we can assume that $\varphi$ has no inessential co-argument and we prove the contrapositive. Let $E_1$ and $E_2$ be distinct elements in $\alg{B}^{*\flat}/{\sim}$ and assume that $\varphi(E_1)\cap (\{i\} \times \alg{A}^{*\flat})\neq \varnothing$ and $\varphi(E_2)\cap (\{i\} \times \alg{A}^{*\flat})\neq \varnothing$ for some $i\leq n$.
Then, the map $\psi \colon \alg{B}^* \rightarrow (n+1)\alg{A}^*$ defined by 
\begin{equation*}
\psi(u) = \begin{cases}
     (n+1, (\pr_2\circ\varphi)(u)) & \text{if } u \in E_1  \\
     \varphi(u) & \text{otherwise } 
   \end{cases}   
\end{equation*}
is a morphism with $[\varphi]_d<_d[\psi]_d$ according to Lemma~\ref{Contin}. Thus we have proved that $[\varphi]_d$ is not maximal.

We prove that the condition is sufficient by contrapositive. Assume that there is a morphism $\psi\colon \alg{B}^* \to (n+1)\alg{A}^*$ with no inessential co-argument and a map $\tau\colon [n+1] \to [n]$ such that $\varphi=\tau_{\alg{A}^*}\circ \psi$. By (\ref{itmdvf01}) there are elements $E_1, E_2$ of $\alg{B}^{*\flat}/{\sim}$ such that
\[
E_1\subseteq \psi^{-1}(\{n+1\}\times \alg{A}^{*\flat}) \quad \text{ and } \quad E_2\subseteq \psi^{-1}(\{\tau(n+1)\}\times \alg{A}^{*\flat}).
\]
It follows that $\varphi(E_i)\cap(\{\tau(n+1)\}\times \alg{A}^{*\flat})\neq \varnothing$ for $i\in\{1,2\}$.
\end{proof}

It follows from Theorem \ref{MorphismsCharacterization} that, if $\alg{B}^*$ has the FCO property and $\alg{B}^{*\flat}/{\sim}=\{E_1, \ldots, E_n\}$, then
\begin{equation}\label{eqn:decomp}
\cate{X}(\alg{B}^*, n\alg{A}^*)/{\equiv_d} \cong\ \cate{X}(\alg{E_1^\#}, \alg{A}^*) \times \cdots \times \cate{X}(\alg{E_n^\#}, \alg{A}^*)\ \cong\ \cate{X}(\alg{B}^*, \alg{A}^*).
\end{equation}
In what follows, we may use these isomorphims without further notice.

In the finite case, we can now improve the characterization of Theorem \ref{MaxDisjUnion} thanks to Theorem \ref{MorphismsCharacterization}. First, we introduce some notation.

\begin{nota}\label{nota:heavy}
Let $\alg{A}$ and $\alg{B}$ be finite algebras of $\var{A}$ and denote by $E_1, \ldots, E_\ell$ the elements of $\alg{B}^{*\flat}/{\sim}$. For any $i\leq \ell$ and any $\varphi \in \cate{X}(\alg{E_i^\#}, \alg{A}^*)$, define $c_\varphi$ by
\[
c_\varphi=\begin{cases}
0 & \text{if } \varphi(\alg{E_i^\#}) \subseteq \mathcal{C}^{\alg{A}^*}\\
1 & \text{else.}
\end{cases}
\]
%\[
%c_\varphi=\begin{cases}
%1 & \text{if } \varphi(\alg{E_i^\#}) \subseteq \mathcal{C}^{\alg{A}^*}\\
%0 & \text{else.}
%\end{cases}
%\]
Denote by $\Pi_{\alg{B}^*\alg{A}^*}$ the Cartesian product $\prod\{\cate{X}(\alg{E_i^\#}, \alg{A}^*) \mid i \leq\ell\}$, and for any $\mbf{\varphi}\in  \Pi_{\alg{B}^*\alg{A}^*}$ set $c_{\mbf{\varphi}}:=c_{\varphi_1}+\cdots +c_{\varphi_\ell}$. We know by \eqref{eqn:decomp} that there is a bijective correspondence between $\cate{X}(\alg{B}^*,\alg{A}^*)$ and $\Pi_{\alg{B}^*\alg{A}^*}$
\end{nota}

\begin{cor}\label{Top-Down}
Let $\alg{A}$ and $\alg{B}$ be finite algebras of $\var{A}$. Using Notation \ref{nota:heavy}, we have
\begin{equation*}
\mbf{\homi}_{\alg{A}\alg{B}} \ \cong\  \biguplus\{\alg{\Pi}^\partial_{c_{\mbf{\varphi}}} \mid \mbf{\varphi}\in\Pi_{\alg{B}^*\alg{A}^*}\}.
\end{equation*}
\end{cor}

\begin{cor}\label{Bottom-up}
Let $\alg{A},\alg{B}$ be finite elements of $\var{A}$ and let $E_1, \dots, E_n$ be the elements of $\alg{A}^{*\flat}/{\sim}$. 
Then 
\[\mbf{\homi}_{\alg{A}\alg{B}} \ \cong \ \biguplus \{\mbf{\Pi}^\partial_{d_\varphi}\mid \varphi\in\cate{X}(\alg{B}^*,\alg{A}^*)\},\]
where $d_\varphi=\#\{E_i \mid \varphi(E_i) \cap \alg{A}^{*\flat}\neq \emptyset\}$ for any $\varphi\in\cate{X}(\alg{B}^*,\alg{A}^*)$.
\end{cor}

We now give a number of applications of Corollary \ref{Top-Down} and \ref{Bottom-up}. 
Recall that a finite algebra $\alg{M}$ is \emph{quasi-primal} if it has the ternary discriminator operation
\[
t^M(x,y,z):=\begin{cases}
x & \text{if } x\neq y\\
z & \text{if } x=y
\end{cases}
\]
as a term function. \emph{Semi-primal} algebras are those quasi-primal algebras that have no isomorphism between their non-trivial subalgebras other than the identity. The finite subalgebras  $\alg{\lucas_n}$ of the standard MV-algebra $[0,1]$ are examples of semi-primal algebras. For any semi-primal algebra $\underline{\alg{M}}$, there is structure $\utilde{\alg{M}}$ that has neither (partial) functions nor  $n$-ary relations for $n\geq 2$ and yields a logarithmic duality for $\classop{ISP}(\underline{\alg{M}})$  \cite[Theorem 3.3.14]{NatDual}.

\begin{prop}\label{prop:tadamm}
Assume that $\underline{\alg{M}}$ is a semi-primal algebra and that $\utilde{\alg{M}}$ is a dualizing structure defined as above. For any finite elements $\alg{A}, \alg{B} \in \var{A}$, we have
\begin{equation*}\label{eqn:mfr}
\mbf{\homi}_{\alg{A}\alg{B}} \cong \quad  \biguplus \{\Pi^\partial_{d_\varphi} \mid \varphi \in \cate{X}(\alg{B}^*, \alg{A}^*)\},
\end{equation*}
where $d_\varphi=|\alg{B}^{*\flat}|-|\varphi^{-1}(\mathcal{C}^{\alg{A}^*})\cap \alg{B}^{*\flat}|$ for every $\varphi \in \cate{X}(\alg{B}^*, \alg{A}^*)$.
\end{prop}

\begin{proof}
It follows from the assumptions on $\utilde{\alg{M}}$ that every singleton is an equivalence class of $\sim$ on $\alg{B}^{*\flat}$. According to Theorem \ref{MorphismsCharacterization}, this means that every morphism $\varphi\colon\alg{B}^* \to \alg{A}^*$ defines a corresponding maximal element $[\varphi']_d$ in $\mbf{\cate{X}}_{\alg{B}^*\alg{A}^*}$  where $\varphi'$ has co-essential arity $d_\varphi$.
\end{proof} 

\begin{exam} \label{ex:lucas}
For every $m\geq 1$, the algebra $\lucas_m$ is semi-primal and a dualizing structure is given in \eqref{eqn:duallucas}. For the sake of illustration, set $m=12$ and consider the algebra $\alg{A}:=\lucas_2\times\lucas_4\times\lucas_4\times \lucas_{6}$ of $\classop{ISP}(\lucas_{12})$. Then $\alg{A}^*$ is the discrete structure 
\[
\struc{\{u_1, \ldots, u_4\}, \{r_m\mid m\in \divi(12)\}, \mathcal{T}_{dis}},
\]
where
\begin{gather*}
r_1=r_3=\varnothing, \quad r_2=\{u_1\}, \quad r_4=\{u_1, u_2, u_3\},\\
 r_6=\{u_1, u_4\}, \quad r_{12}=\{u_1, \ldots, u_4\}.
\end{gather*}
It follows that morphisms $\varphi\colon \alg{A}^*\to \alg{A}^*$ are defined by the constraints
\[
\varphi(u_1)=u_1, \quad \{\varphi(u_2), \varphi(u_3)\}\subseteq \{u_1, u_2, u_3\}, \quad \varphi(u_4)\in\{u_1, u_4\},
\]
and $d_\varphi=|\alg{A}^*|=4$ for every $\varphi$ since there is no constant in $\utilde{\lucas}_m$. We obtain by Proposition \ref{prop:tadamm} that the minor homomorphism poset $\mbf{\mathcal{MV}_\alg{A}}$ is the disjoint union of 18 copies of $\Pi^\partial_4$.
\end{exam}

The argument developed in Example \ref{ex:lucas} leads us to the following result.
\begin{prop}
Let $\{p_1, \ldots, p_k\}$ and $\{q_1, \ldots, q_k\}$ be two sets of prime numbers and $\alpha_1, \ldots, \alpha_k \geq 1$. Set $m:=p_1^{\alpha_1} \times \cdots \times p_k^{\alpha_k}$ and $m':={q_1}^{\alpha_1} \times \cdots \times q_k^{\alpha_k}$, and define the map $\mu\colon \divi(m) \to \divi(m')$ by 
\[
\mu(p_1^{\beta_1} \times \cdots \times p_k^{\beta_k})=q_1^{\beta_1} \times \cdots \times q_k^{\beta_k}
\]
Then, for every $m_1, \ldots, m_t\in \divi(m)$, it holds that
\[
\mbf{\mathcal{MV}}_{\lucas_{m_1}\times \cdots \times \lucas_{m_t}} \ \cong \ 
\mbf{\mathcal{MV}}_{\lucas_{\mu(m_1)}\times \cdots \times \lucas_{\mu(m_t)}}.
\]
\end{prop}

\begin{exam}
Corollary \ref{Top-Down} helps to define a systematic technique to compute $\mbf{\var{D}}_{\alg{L}\alg{L}'}$ for finite distributive lattices  $\alg{L}$ and $\alg{L}'$. Indeed, it turns out that ${\alg{L}'}^{*\flat}/{\sim}$ is the set $\{E_1, \ldots, E_\ell\}$ of the connected components of the Hasse diagram of ${\alg{L}'}^{*\flat}$ considered as an undirected graph. Maximal elements in $\mbf{\cate{X}}_{\alg{L}^*{\alg{L}'}^{*}}$ correspond via \eqref{eqn:decomp} to tuples $\mbf{\varphi}:=(\varphi_1, \ldots, \varphi_\ell)$ such that $\varphi \in \cate{X}(\alg{E_i^\#}, \alg{L}^*)$ for every $i\leq \ell$, that is, to tuples of partial morphisms on ${\alg{L}'}^*$ with maximal domains. Each such tuple corresponds to a maximal element of co-essential arity $c_{\mbf{\varphi}}$ and we obtain
\[
\mbf{\cate{X}}_{\alg{L}^*{\alg{L}'}^{*}}\ \cong \ \biguplus\{\alg{\Pi}^\partial_{c_{\mbf{\varphi}}} \mid \varphi \in \cate{X}(\alg{E_1^\#}, \alg{L}^*)\times \cdots \times \cate{X}(\alg{E_\ell^\#}, \alg{L}^*) \}
\]
by Corollary \ref{Top-Down}.
\end{exam}

If $\alg{B}$ is a Boolean algebra, then its $\{\vee, \wedge\}$-reduct $\alg{B}^-$ is a distributive lattice. It turns out that these reducts, which form the class of \emph{complemented distributive lattices}, can be recognized by their minor homomorphim posets, as shown in the next result. First, recall that the join-irreducible elements of a complemented distributive lattice coincide with its atoms (if there is an atom $a$ strictly below a join irreducible element $b$ and $b^c$ is the complement of $b$ then $\{0,a, b, b^c, 1\}$ is isomorphic to $N_5$).

\begin{prop}\label{prop:complat}
A finite distributive lattice $\alg{L}$ is complemented if and only if $\mbf{\var{D}}_{\alg{L}}$ is isomorphic to %$\biguplus\{\biguplus\{\alg{\Pi}_j^\partial\mid 1\leq i \leq d_j\} \mid 0\leq j \leq n\}$ 
\[
\biguplus_{1\leq j\leq n}\biguplus\{\alg{\Pi}_j^\partial\mid 1\leq i \leq d_j\} 
\]
for some $n\geq 1$, where $d_j:=\binom{n}{j}n^j 2^{n-j}$ for every $j\leq n$.
\end{prop}
\begin{proof}
Let $\alg{L}$ be a finite complemented distributive lattice and $a_1, \ldots, a_n$ be its atoms. Then $\alg{L}^*$ is isomorphic to the antichain $\{a_1, \ldots, a_n\}$ together with the constants $1^{\alg{L}^*}=L$ and $0^{\alg{L}^*}=\varnothing$ as top and bottom element, respectively. It follows that the equivalence $\sim$ on $\alg{L}^{*\flat}$ is the identity, just as in the Stone dual of the Boolean algebra associated with $\alg{L}$. But, on the contrary to the case of Boolean algebras where the dual structures have no constants, morphisms can map elements $a_i$ to constants $0$ or $1$, which influences the number of morphisms and their co-essential arities. 

Let $j\leq n$ and let us count the number of morphisms $\varphi\colon \alg{L}^* \to j\alg{L}^*$ with co-essential arity $j$ such that $[\varphi]_d$ is maximal. Such a morphism $\varphi$ maps $j$ elements among ${a_1, \ldots, a_n}$ to distinct copies of $\alg{L}^{*\flat}$ in $j\alg{L}^*$ and $(n-j)$ elements to constants $0$ or $1$. So, there are $d_j:=\binom{n}{j}n^j 2^{n-j}$ such morphisms and each of them  satisfies $d_{{\varphi}}=j$. The conclusion follows from Corollary \ref{Bottom-up}.

Conversely, assume that $\alg{L}$ is a distributive lattice whose homomorphism poset contains exactly $d_j$ disjoint copies of $\alg{\Pi}_j^\delta$ for each $j\in\{1, \ldots, n\}$. It follows from Theorem \ref{MaxDisjUnion} that the maximum co-essential arity of an element $\varphi\in \cate{X}_{\alg{L}^*}$ is $n$  and that the number of non $\equiv_d$-equivalent morphisms $\phi\colon \alg{L}^* \to n\alg{L}^*$ of co-essential arity $n$ is equal to $d_n=n^n$. Moreover, we have $|\alg{L}^{*\flat}/{\sim}|=n$ by Proposition \ref{prop:bbound}. 
Now, we prove that the elements $E_1, \ldots, E_n$ of $\alg{L}^{*\flat}/{\sim}$ are singletons. For the sake of contradiction, assume that $E_1$ contains two elements $u$ and $v$. By definition of $\sim$, we may assume $u< v$. The map $\varphi_1\colon \alg{L}^* \to n\alg{L}^*$ defined by 
\[
\varphi_1(z)=
\begin{cases}
0 & \text{if } z\leq u\\
(1,z) & \text{if } z\in E_1 \text{ and } z\not\leq u\\
(i,z) & \text{if } i\neq 1 \text{ and } z \in E_i,
\end{cases}
\]
is a morphism of co-essential arity $n$. Now, for every $i\leq n$, let $e_i$ be a fixed element of $E_i$. For every $h\in [n]^{[n]}$, define the map $\varphi_h\colon \alg{L}^* \to n\alg{L}^*$ by $\varphi_h(E_i)=\{(i, e_{h(i)})\}$. 
 Together with $\varphi_1$, we have found $n^n+1$ pairwise non $\equiv_d$-equivalent morphisms $\alg{L}^* \to n\alg{L}^*$ of co-essential arity $n$, a contradiction. 
\end{proof}

Finite median algebra are handled similarly as finite distributive lattices. 
\begin{prop}\label{prop:medmed}
Let $\alg{A}$ be a median algebra.
\begin{enumerate}
\item\label{bgf:01} If $\alg{A}$ is the $\me$-reduct of a distributive lattice, then $\ess(f)\leq 1$ for every $f\in{\var{M}}_{\alg{A}}$.
\item\label{bgf:02} $\alg{A}$ is the $\me$-reduct of a finite Boolean algebra if and only if 
\[
\mbf{\var{M}}_{\alg{A}}\cong \biguplus_{1\leq i \leq k}\biguplus\{\mbf{\Pi}^\partial_i\mid j\leq \binom{k}{i}2^k k^i \} 
\] holds for some $k\geq 1$.
\end{enumerate}
\end{prop} 
\begin{proof}
(\ref{bgf:01}) If $\alg{A}$ is the $\me$-reduct of a distributive lattice $\alg{L}=\struc{A, \vee, \wedge}$, then the prime convex subsets of $\alg{A}$ are $A$, $\varnothing$ and the prime filters and prime ideals of $\alg{L}$. It follows that $\alg{A}^*$ is the disjoint union of the poset of prime filters and the poset of prime ideals of $\alg{L}$, with $\varnothing$ and $A$ as additional bottom and top elements, respectively, and where $u^c=A\setminus u$ for any $u\in \alg{A}^*$. It follows that $\sim$ is the total equivalence on $\alg{A}^{*\flat}$, and $\ess_d(\varphi)\leq 1$ for any morphism $\varphi\colon \alg{A}^* \to n\alg{A}^*$.

(\ref{bgf:02}) If $\alg{A}$ is the $\me$-reduct of the Boolean algebra $\alg{B}=\alg{2}^k$ then it follows by (\ref{bgf:01}) that $\alg{A}^*$ is an antichain $u_1, \ldots, u_k, u_1^c, \ldots, u_k^c$ with an additional top and bottom element $1^{\alg{A}^*}=A$ and $0^{\alg{A}^*}=\varnothing$ , respectively. It follows that $\alg{A}^{*\flat}/{\sim}=\{\{u_i, u_i^c\} \mid i\leq k\}$ has cardinality $k$. For each $i\leq k$, there are $2k$ morphisms $\varphi_i \colon \{u_i, u_i^c\}^\# \to \alg{A}^*$ that satisfy $\im(\varphi_i)\neq\{A, \varnothing\}$, and two morphisms $\varphi_i \colon \{u_i, u_i^c\}^\# \to \alg{A}^*$ that satisfy $\im(\varphi_i)=\{A, \varnothing\}$. We conclude that for any $0\leq i \leq k$ there are $\binom{k}{i}(2k)^i2^{k-i}=\binom{k}{i}2^k k^i$ tuples $\mbf{\varphi}\in \Pi_{\alg{A}^*\alg{A}^*}$ with $c_{\mbf{\varphi}}=i$, and we conclude the proof by Corollary \ref{Top-Down}.

Conversely, assume that the minor homomorphism poset of $\alg{A}$ is made of $\binom{k}{i}2^k k^i$ copies of $\alg{\Pi}_i^\partial$ for every $i\in\{0, \ldots, k\}$. In particular, the maximum essential arity of an element of $\mbf{\var{M}}_{\alg{A}}$ is $k$, which implies that $\alg{A}^{*\flat}/{\sim}$ has $k$ elements $E_1$, \ldots, $E_k$. It follows that $\alg{A}^*\cong \alg{E_1^\#}\oplus \cdots \oplus \alg{E_k^\#}$, so that $(\alg{A}^*)_*\cong (\alg{E_1^\#})_*\times \cdots \times (\alg{E_k^\#})_*$. Moreover, there are $2^k$ tuples $\mbf{\varphi}\in \Pi_{\alg{A}^*\alg{A}^*}$ that satisfy $c_{\mbf{\varphi}}=0$, which means that 
\begin{equation}\label{eqn:abacus}
2^k=|(\alg{A}^*)_*|=|(\alg{E_1^\#})_*|\times \cdots \times |(\alg{E_k^\#})_*|
\end{equation}
according to \eqref{eqn:decomp}. Since $|E_i|>1$ for every $i\leq k$, we obtain that $|(\alg{E}_1^\#)_*|=\cdots=|(\alg{E}_k^\#)_*|=2$. It follows that $\alg{E}_i^\#$ is the dual of the $\me$-reduct of the 2 element Boolean algebra, so that $\alg{A}$ is the $\me$-reduct of the $2^k$ element Boolean algebra.
\end{proof}

Example \ref{ex:fini-cofini} shows that  the homomorphism poset of an infinite algebra can get pretty wild. We end the section with an additional example in that direction.

\begin{exam}\label{ex:bool2}
\emph{Let $\alg{\free}_\omega$ be the free Boolean algebra with countably many generators. For any $n\geq 1$ and any $f\in \homi_{\alg{\free}_\omega}$ with essential arity $n$, there are (countably) infinitely many elements $f'\colon  \free^{n+1}_\omega  \to \free_\omega$ such that $f \prec f'$.}
\begin{proof}
The Stone dual of $\alg{\free}_\omega$ is the Cantor space $\Gamma$. Since $f^*\colon \Gamma \to n\Gamma$ has co-essential arity $n$, we know that $Y:=(f^{*})^{-1}(\{n\}\times \Gamma)$ is a nonempty clopen subset of $\Gamma$, so it is homeomorphic to $\Gamma$. Let $\{\omega_i \mid i \in \N\}$ be a countable clopen basis of $Y$.
For every $i\in \N$, the map $\varphi\colon \Gamma \to (n+1)\Gamma$ defined as 
\[
\varphi(u)=\begin{cases}
f^*(u) & \text{if } u\not\in Y  \text{ or } u\not\in \omega_i\\
\big(n+1,\pr_2(f^*(u))\big) & \text{if } u \in \omega_{i}
\end{cases}
\]
is a morphism of co-essential arity $n+1$ and $f^*\prec_d \varphi$.
\end{proof}
\end{exam}

\section{Concluding remarks and further research}
In this paper, we used natural duality theory to investigate the minor relation for algebra homomorphisms. Although our developments are limited to finitely generated quasivarieties that admit a logarithmic duality, we have shown that natural duality theory may turn to be a powerful tool to explore combinatorial problems pertaining to general algebra. Conversely, note that Proposition \ref{prop:tadam}, Corollary \ref{cor:tidim} and Theorems  \ref{thm:weakly} and \ref{MaxDisjUnion} can be used as criteria to test non-dualizability of a finite algebra $\underline{\alg{M}}$ as follows.

\begin{prop} Let $\underline{\alg{M}}$ be a finite algebra. If one of the following conditions is satisfied, then no structure $\utilde{\alg{M}}$ can yield a logarithmic duality for $\var{A}=\classop{ISP}(\underline{\alg{M}})$.
\begin{enumerate}
\item There is a finite algebra $\alg{A}$ in $\var{A}$ whose minor homomorphism poset is not isomorphic to a disjoint union of order dual of partition lattices.
\item There are an algebra $\alg{A}$ in $\var{A}$ and homomorphims $f,g\colon \alg{A}^n \to \alg{A}$ for some $n\geq 2$ such that $\deck(f)=\deck(g)$ but $f\not\equiv g$.
\item There are finite algebras $\alg{A}, \alg{B}$ in $\var{A}$ and a homomorphism $f\colon \alg{A}^n \to \alg{B}$ such that $\ess(f)=n\geq 2$ and $f\circ\delta_I \equiv f\circ\delta_J$ for some $I\neq J$ in $\binom{n}{2}$.
\item There are algebras $\alg{A}, \alg{B} \in \var{A}$ and a homomorphism $f\colon \alg{A}^n \to \alg{B}$ for some $n\geq 2$ whose arity gap is not 1. 

\end{enumerate}
\end{prop}

We now list topics for further research. Theorem \ref{MaxDisjUnion} states that for any $\alg{A}, \alg{B}\in \var{A}$ such that $\ess(\var{A}_{\alg{A}\alg{B}})$ is bounded, the poset $\mbf{\var{A}}_{\alg{A}\alg{B}}$ is a disjoint union of finite partition lattices, and leads us to the following definition.
\begin{defin}
Let $\alg{A}, \alg{B}$ be elements of $\var{A}$ such that $\ess(\var{A}_{\alg{A}\alg{B}})$ is bounded, and let $K$ be the set of maximal elements of $\mbf{\var{A}}_{\alg{A}\alg{B}}$. The \emph{$(\alg{A}, \alg{B})$-minor sequence  $s_{\alg{A}\alg{B}}$}  is defined by $s_{\alg{A}\alg{B}}(1)=\#\{[f]\in K \mid \ess(f)\in\{0,1\}\}$ and $s_{\alg{A}\alg{B}}(n)=\#\{[f]\in K \mid \ess(f)=n\}$ for every $n\geq 2$ (these cardinals may be infinite). If $\alg{A}=\alg{B}$, we write $s_{\alg{A}}$ for $s_{\alg{A}\alg{A}}$, and we call it the \emph{minor sequence of $\alg{A}$}.
\end{defin}
 If $\ess(\var{A}_{\alg{A}\alg{B}})$ is bounded, then the sequence $s_{\alg{A}\alg{B}}$ completely characterizes $\mbf{\var{A}}_{\alg{A}\alg{B}}$, which leads us to the following problems.
\begin{enumerate}[(I)]
\item Characterize those cardinal sequences that can be realized as $(\alg{A}, \alg{B})$-minor sequence for some $\alg{A}, \alg{B} \in \var{A}$.
\item Proposition \ref{prop:complat} states that the class of finite complemented lattices can be characterized in the variety of distributive lattices by their minor sequence. Proposition \ref{prop:medmed} states that finite ternary Boolean algebras can be characterized among median algebras by their minor sequences. Generally, how to find subclasses of $\var{A}$ that can be characterized in $\var{A}$ by the minor sequences of its elements?
\item More generally, for any cardinal $\alpha$ and any subclass $\var{B}$ of $\var{A}$, say that $\var{B}$ is \emph{$\alpha$-minor determined} if for every $\alg{A}\in \var{B}$, every set of mutually non-isomorphic $\alg{B}\in \var{B}$ such that $\mbf{\var{A}}_\alg{B}\cong \mbf{\var{A}}_\alg{A}$ has cardinality smaller than $\alpha$. For example, the class of finite Boolean algebras is $1$-minor determined, but if $m>1$ then the class of finite MV${}_m$-algebras is not (if $p$ is a prime divisor of $m$ then the minor posets of $\mbf{\lucas}_p$ and $\mbf{\lucas}_1$ are isomorphic).
What are other nontrivial examples of $\alpha$-minor determined classes $\var{B}$? A similar problem has been investigated in \cite{Koubek1994} by considering endomorphism monoids instead of minor posets.
\end{enumerate}

We obtain other interesting open problems by going beyond logarithmic natural dualities.
\begin{enumerate}[(I)]\setcounter{enumi}{3}
\item Find instances of non logarithmic natural dualities for which co-products can still be easily computed in the dual category, and apply the tools developed in this paper.
\item Use the TwoSwap Theorem \cite[Theorem 2.4]{Davey2012} to study minor continuous homomorphism posets in topological algebras.
\item Use other types of dualities (Pontryagin duality, De Vries duality,\ldots) to study minor homomorphism posets.
\end{enumerate}

\section*{Acknowledgments}
The first author is supported by the Luxembourg National Research Fund under the project  PRIDE17/12246620/GPS.

\section*{Data availability statement}
Data sharing not applicable to this article as no datasets were generated or analysed during the current study.

\end{document}